\documentclass[11pt,letter]{article} %
\usepackage[papersize={8.5in,11in},margin=1in]{geometry}
\usepackage[format=hang,font={small}]{caption}

\usepackage{microtype}
\usepackage{graphicx}
\usepackage{subfigure}
\usepackage{booktabs} %

\usepackage{multicol}

\usepackage[numbers,sort,compress]{natbib}

\usepackage[noblocks]{authblk}

\usepackage{times}
\usepackage{microtype}
\usepackage{graphicx}
\usepackage{subfigure}
\usepackage[pagebackref,linktocpage]{hyperref}

\hypersetup{
    colorlinks=true,%
    citecolor=blue,%
    filecolor=blue,%
    linkcolor=blue,%
    urlcolor=blue
}

\usepackage{color}
\usepackage{url}            %
\usepackage{amsfonts}       %
\usepackage{nicefrac}       %
\usepackage{microtype}      %

\usepackage{multirow}

\usepackage{times}
\usepackage{algorithm}
\usepackage{xspace}

\usepackage{amsfonts}
\usepackage{graphicx}
\usepackage{amsmath}
\usepackage{amssymb}
\usepackage{bbding}
\usepackage{pifont}
\usepackage{wasysym}
\usepackage{bm}
\usepackage{algorithmic}
\usepackage{epstopdf}
\usepackage{psfrag}
\usepackage{bbm}
\usepackage{enumerate}
\usepackage{subeqnarray}
\usepackage{cases}

\newcommand{\bbR}{\mathbb{R}}

\newcommand{\1}{{\mathbbm{1}}}

\renewcommand{\b}{{\mathbf{b}}}

\renewcommand{\v}{{\mathbf{v}}}

\usepackage{color}

\usepackage{soul}

\newcommand{\beq}{\begin{equation}}
\newcommand{\eeq}{\end{equation}}
\newcommand{\beqa}{\begin{eqnarray}}
\newcommand{\eeqa}{\end{eqnarray}}
\newcommand{\beqas}{\begin{eqnarray*}}
\newcommand{\eeqas}{\end{eqnarray*}}
\newcommand{\bi}{\begin{itemize}}
\newcommand{\ei}{\end{itemize}}
\newcommand{\ba}{\begin{array}}
\newcommand{\ea}{\end{array}}

\newtheorem{theorem}{Theorem}
\newtheorem{lemma}{Lemma}
\newtheorem{corollary}{Corollary}
\newtheorem{remark}{Remark}

\newtheorem{assumption}{Assumption}

\usepackage{amsmath,amsfonts,bm}

\def\1{\bm{1}}

\def\vzero{{\bm{0}}}

\def\vb{{\bm{b}}}
\def\vc{{\bm{c}}}

\def\vs{{\bm{s}}}

\def\vu{{\bm{u}}}
\def\vv{{\bm{v}}}
\def\vw{{\bm{w}}}
\def\vx{{\bm{x}}}
\def\vy{{\bm{y}}}
\def\vz{{\bm{z}}}

\def\mA{{\bm{A}}}
\def\mB{{\bm{B}}}

\def\mI{{\bm{I}}}

\def\mQ{{\bm{Q}}}

\DeclareMathAlphabet{\mathsfit}{\encodingdefault}{\sfdefault}{m}{sl}
\SetMathAlphabet{\mathsfit}{bold}{\encodingdefault}{\sfdefault}{bx}{n}

\def\gV{{\mathcal{V}}}
\def\gW{{\mathcal{W}}}
\def\gX{{\mathcal{X}}}

\DeclareMathOperator*{\argmin}{arg\,min}

\author[*]{Chaobing Song}
\author[*]{Yong Jiang}
\author[$\dagger$]{Yi Ma}
\affil[*]{Tsinghua-Berkeley Shenzhen Institute, Tsinghua University}
\affil[$\dagger$]{EECS Department, University of California, Berkeley}

\begin{document}
\title{Breaking the $O(1/\epsilon)$  Optimal Rate for a Class of Minimax Problems\footnote{\textbf{Funding: } This work was done during Chaobing Song's visit to Professor Yi Ma's group at UC Berkeley. The work is partially supported by the TBSI program and EECS Startup fund of Professor Yi Ma.}}
\maketitle

\begin{abstract}
It is known that for convex optimization $\min_{\vw\in\gW}f(\vw)$, the best possible rate of first order accelerated methods is $O(1/\sqrt{\epsilon})$.
However, for the bilinear minimax problem: $\min_{\vw\in\gW}\max_{\vv\in\gV}f(\vw)+\langle\vw, \mA\vv\rangle-h(\vv)$ where both $f(\vw)$ and $h(\vv)$ are convex, the best known rate  of first order methods slows down to $O(1/{\epsilon})$. It is not known whether one can achieve the accelerated  rate $O(1/\sqrt{\epsilon})$ for the bilinear minimax problem without assuming $f(\vw)$ and $h(\vv)$ being strongly convex. In this paper, we fill this theoretical gap by proposing a bilinear accelerated extragradient (BAXG) method. We show that when $\gW=\bbR^d$, $f(\vw)$ and $h(\vv)$ are convex and smooth, and $\mA$ has full column rank, then the BAXG method achieves an accelerated rate $O(1/\sqrt{\epsilon}\log \frac{1}{\epsilon})$, within a logarithmic factor to the likely optimal rate $O(1/\sqrt{\epsilon})$. As result, a large class of bilinear convex concave minimax problems, including a few problems of practical importance, can be solved much faster than previously known methods.
\end{abstract}

\section{Introduction}

\subsection{Background}
Nesterov's acceleration is a core technique to improve the convergence behavior of first order methods for a convex optimization problem: 
\vspace{-2mm}
\begin{equation}
\min_{\vw\in\gW} f(\vw),\label{eq:smo-con}
\vspace{-1mm}
\end{equation}
where $\gW$ is a simple closed convex set and admits an efficient projection operator, and $f(\vw)$ is convex and smooth. Let $\vw^*$ be an optimal solution of \eqref{eq:smo-con}. Then to find an $\epsilon$-accurate solution $\vw\in\gW$ such that $f(\vw)-f(\vw^*)\le \epsilon,$ gradient descent (GD) methods need at least $O(1/\epsilon)$ iterations. However, by combining a momentum step, Nesterov's accelerated gradient descent (AGD) can improve the $O(1/\epsilon)$ rate to $O(1/\sqrt{\epsilon})$, which is optimal and can not be improved further by first order methods that only access gradient information $\nabla f(\vw)$, $\forall \vw\in\gW$.       

In practice, the optimization often is subject to certain (linear or affine) constraints and the associated Lagrangian formulation typically leads to a (bilinear) minimax problem. So in general, we may need to consider the following bilinear convex concave minimax problem:
\vspace{-1mm}
\begin{align}
\min_{\vw\in\gW}\max_{\vv\in\gV}f(\vw) +\langle \vw, \mA\vv\rangle - h(\vv),\label{eq:con-con}
\vspace{-1mm}
\end{align}
where $\gW\subset\bbR^d, \gV\subset\bbR^n$ are simple closed convex sets and admit efficient projection operators respectively, $f(\vw)$ and $h(\vv)$ are convex and smooth. To attain an $\epsilon$-accurate solution $(\vw, \vv)$ in terms of a proper  merit function (See \eqref{eq:appro-weak}), without using the Nesterov's acceleration, the extragradient method \cite{nemirovski2004prox} and the primal dual method \cite{chambolle2011first} need at most $O(1/\epsilon)$ iterations. According to  \cite{nemirovsky1983problem}, \cite{nemirovski2004prox}, for general bilinear convex concave minimax problems, the $O(1/\epsilon)$ rate is the optimal convergence rate of first-order methods. 

In the more recent work of \cite{ouyang2018lower}, to explicitly establish the $O(1/\epsilon)$ rate as the lower bound of the problem \eqref{eq:con-con}, the authors have constructed a particular instance of \eqref{eq:con-con} as follows 
\vspace{-1mm}
\begin{align}
\min_{\vw\in\bbR^d} f(\vw) = \frac{1}{2}\vw^T \mB\vw - \vc^T\vw, \; 
\mbox{s.t.}\; \mA^T\vw = \vb, \label{eq:con-con-exam}
\vspace{-1mm}
\end{align}
where $\vc\in\bbR^d, \mB\in\bbR^{d\times d}$ is a symmetric positive semidefinite matrix,  $\mA\in\bbR^{d\times n}$ is a matrix with full column rank and $\vb\in\bbR^n$. %
The problem \eqref{eq:con-con-exam} is equivalent to its Lagrangian formulation
\vspace{-1mm}
\begin{align}
\min_{\vw\in\bbR^d}\max_{\vv\in\bbR^n}\frac{1}{2}\vw^T \mB\vw - \vc^T \vw +\langle \vw, \mA\vv\rangle - \vb^T\vv, \label{eq:lag}
\vspace{-1mm}
\end{align}
which is a particular instance of \eqref{eq:con-con}.  
In \cite{ouyang2018lower}, first order methods for \eqref{eq:con-con-exam} are meant by methods that access the first order oracle that returns
\vspace{-1mm}
\begin{align}
(\nabla f(\vw), \mA^T\vw, \mA\vv), \forall \vw\in\bbR^d,\vv\in\bbR^n.\label{eq:oracle}
\vspace{-1mm}
\end{align}
Let $\vw^*$ be a solution of \eqref{eq:con-con-exam}. \cite{ouyang2018lower} has shown that to attain an $\epsilon$-accurate solution such that $|f(\vw)-f(\vw^*)|\le \epsilon, \|\mA^T\vw-\vb\|\le \epsilon$, at least $O(1/\epsilon)$ iterations are needed if we can only access the first order oracle specified in \eqref{eq:oracle-2}. Therefore, \cite{ouyang2018lower} concludes that first-order methods on affinely constrained problems generally cannot be accelerated from the known convergence rate $O(1/\epsilon)$ to $O(1/\sqrt{\epsilon})$. 

Despite this result, in this paper, we show that somewhat surprisingly, for the problem in \eqref{eq:con-con-exam}, if we are allowed to access the following extended first order oracle that returns 
\begin{align}
(\nabla f(\vw), \mA^T\vw, \mA\vv, \mA^T\mA\vv), \forall \vw\in\bbR^d,\vv\in\bbR^n,\label{eq:extend-oracle}
\end{align}
then {\em the lower bound $O(1/{\epsilon})$ can be broke down to $O(1/\sqrt{\epsilon}\log \frac{1}{\epsilon})$}, since $\mA^T\mA$ is a positive definite matrix by the assumption of $\mA$ being full column rank. Unlike \cite{ouyang2018lower}, we do not simply establish a tighter lower bound with the extended first order oracle  \eqref{eq:extend-oracle}. Instead, we explicitly construct a bilinear accelerated extragradient (BAXG) method and characterize its complexity for solving the class of bilinear convex concave problems:
\vspace{-1mm}
\begin{align}
\min_{\vw\in\bbR^d}\max_{\vv\in\gV}f(\vw) +\langle \vw, \mA\vv\rangle - h(\vv),\label{eq:con-con-prob}
\vspace{-2mm}
\end{align}
where $\gV\subset\bbR^n$ is a closed convex set, both $f(\vw)$ and $h(\vv)$ are  convex and smooth,
and $\mA\in\bbR^{d\times n}$ is a matrix with full column rank. Clearly, by \eqref{eq:lag}, \eqref{eq:con-con-exam} is a special case of \eqref{eq:con-con-prob}.
Many well-known problems in machine learning can be reduced to \eqref{eq:con-con-prob} (see discussion in Section \ref{eq:moti-exam}). 

If $\gV\subset\bbR^n$ is a general convex set instead of the simple $\bbR^n$, we typically assume the first order oracle can access the projection onto this set.\footnote{In other words, projection onto $\gV\subset\bbR^n$ can be computed with cost on par with computing the gradient.}  For \eqref{eq:con-con-prob}, correspondingly, the first order oracle \eqref{eq:oracle} becomes as 
\begin{align}
\!\!\!\!(\nabla f(\vw), \nabla g(\vv), \text{Prj}(\vv),  \mA^T\vw, & \mA\vv), 
\forall \vw\in\bbR^d,\vv\in\gV, \label{eq:oracle-2}
\vspace{-1mm}
\end{align}
where $\text{Prj}(\vv)$ denotes the Euclidean projection operator 
\begin{align}
\text{Prj}(\vv) = \argmin_{\hat{\vv}\in\gV}\|\hat{\vv} - \vv\|^2,\nonumber
\end{align}
and the extended first order oracle \eqref{eq:extend-oracle} becomes 
\begin{align}
(\nabla f(\vw), \nabla g(\vv), \text{Prj}(\vv), \mA^T\vw, & \mA\vv, \mA^T\mA\vv),
\forall \vw\in\bbR^d,\vv\in\gV. \label{eq:extend-oracle-2}
\end{align}

\subsection{Formulation}
We first reformulate the problem \eqref{eq:con-con-prob} under the variational inequality framework \cite{facchinei2007finite}. To simplify notation, we define 
$g(\vx) \doteq f(\vw) + h(\vv)$ and denote
\begin{equation}
\gX\doteq\bbR^d\times \gV,\quad
\vx \doteq
\left[
\begin{matrix}
     \vw  \\
     \vv
\end{matrix}
\right],\quad
\mQ \doteq\left[\begin{matrix}
   \boldsymbol 0  &  \mA \\
-\mA^T &  \boldsymbol 0
\end{matrix}\right], \label{eq:H}
\end{equation}
where $g(\vx)$ is the potential function and $\mQ$ is the operator about the interacting term between $\vw$ and $\vv.$ By the optimality condition of \eqref{eq:con-con-prob}, solving \eqref{eq:con-con-prob} is equivalent to finding a solution $\vx^*$ for a mixed variational inequality problem, denoted as MVIP$(\nabla g(\vx) + \mQ, \gX)$, such that $\forall \vy\in\gX,$
\begin{align}
g(\vy) - g(\vx^*) + \langle \mQ\vx^* , \vy - \vx^*\rangle \ge 0. \label{eq:strong}
\end{align}

As shown in \cite{he20121}, in the convex concave setting, the condition \eqref{eq:strong} is equivalent to $\forall \vy\in\gX,$ 
\begin{align}
g(\vy) - g(\vx^*) + \langle \mQ\vy, \vy - \vx^*\rangle \ge 0.   \label{eq:weak}
\end{align}
As a result, similar to \cite{monteiro2011complexity,he20121,gu2014customized,chen2017accelerated}, in this paper, we aim to find an $\epsilon$-accurate solution $\tilde{\vx}$ of the MVIP$(\nabla g(\vx) + \mQ, \gX)$ problem such that $\forall \vy\in\gX,$   
\begin{align}
 g(\tilde{\vx}) - g(\vy)  + \langle \mQ\vy,  \tilde{\vx} - \vy\rangle \le \epsilon.   \label{eq:appro-weak}
\end{align}
If the problem \eqref{eq:con-con-prob} has additional structures, then better problem-specific merit functions can be used for algorithm design. For instance, for the particular instance \eqref{eq:con-con-exam}, one can use  $|f(\vw)-f(\vw^*)|\le \epsilon, \|\mA^T\vw-\vb\|\le \epsilon$ as merit function. In this paper,  we mainly focus on establishing the substantial improvement of convergence rate hence, to simplify analysis, will use \eqref{eq:appro-weak} to measure the progress of our algorithm.  

Our analysis will be based on the following assumptions.%
\begin{assumption}\label{ass:g}
Given the closed convex set $\gX = \bbR^d\times \gV,$
$g(\vx)$ is convex and smooth such that $\forall \vx, \vy\in \gX,$ 
\begin{align}
&g(\vx)\ge g(\vy) + \langle \nabla g(\vy), \vx-\vy\rangle,\\
&\|\nabla g(\vx)  - \nabla (\vy)\|\le L \|\vx-\vy\|,  
\end{align}
where $L>0$ is the smoothness constant of $g(\vx).$
\end{assumption}
\begin{assumption}\label{ass:A}
The matrix $\mA\in\bbR^{d\times n}$ has full column rank with 
the least and largest nonzero singular values $\sigma_{\min}>0$ and  $\sigma_{\max}>0$, respectively. So we have 
\begin{align}
\boldsymbol 0 \prec \sigma_{\min}^2\mI \preceq \mA^T\mA \preceq \sigma^2_{\max} \mI.
\end{align}
\end{assumption}
Assumption \ref{ass:g} is a standard assumption about convexity and smoothness of $g(\vx)$, which in turn is about $f(\vw)$ and $h(\vv).$  In Assumption \ref{ass:A}, the full column rank assumption of $\mA$ means that in our setting $n\le d$ and
is the key to obtain the accelerated rate $O(1/\sqrt{\epsilon}\log \frac{1}{\epsilon})$.  

\subsection{Technical Novelty} To break through the lower bound $O(1/\epsilon)$ and obtain the accelerated rate $O(1/\sqrt{\epsilon}\log \frac{1}{\epsilon})$, we propose a new algorithm, known as bilinear accelerated extragradient (BAXG), which has two loops: the outer loop uses the Nesterov's acceleration trick for the potential function $g(\vx)$, while the inner loop uses the Nesterov's acceleration trick for the subproblem involving the interacting term about $\mQ.$ By summing the total number of inner iterations, we obtain the desired complexity result. The BAXG method relies on the following three key technical novelties.

The first key is that we consider the Nesterov's acceleration strategy based on \emph{approximate backward Euler discretization.} The difference between forward Euler discretization and approximate backward Euler discretization can be found in \cite{diakonikolas2018accelerated,diakonikolas2019approximate}. The original acceleration strategy \cite{nesterov1998introductory} for first order methods is based on forward Euler discretization where we only need to evaluate one gradient in each iteration but it can not be generalized to high order methods.  Approximate
backward Euler discretization is original designed for accelerating high-order methods such as the accelerated cubic regularized Newton (ACNM) \cite{nesterov2008accelerating,song2019towards} method. Then it is found applications in first order methods for designing a variant \emph{accelerated extragradient descent} (AXGD) \cite{diakonikolas2018accelerated} of AGD, for convex minimization problems. The proposed BAXG method in this paper follows the same paradigm of AXGD for the potential function $g(\vx)$, but
does not make approximations about the interacting term $w.r.t.$ $\mQ$. With this treatment, the convergence rate in terms of the number of outer iterations is not affected by the interacting term and maintains the same rate $O(1/\sqrt{\epsilon})$ as the AXGD method. However, as a tradeoff, in each iteration of the BAXG method, we must solve an  easier but nontrivial bilinear convex concave minimax subproblem to certain accuracy.

The second key is that because $\vw\in\bbR^d$, the minimax subproblem can be equivalently converted into a {\em strongly convex minimization} subproblem.

\begin{table*}[t!]
\centering
\caption{Comparison of complexity results for solving the problem \eqref{eq:con-con-prob}. (In the complexity bound of the proposed BAXG method, $\tilde{O}$ hides the logarithmic factors about $\sigma_{\min}, \sigma_{\max}$ and $L.$)}\label{tb:result}
\begin{tabular}{|c|c|c|c|}
\hline
Algorithm                                  &          Complexity Bound          &       Acceleration (w/o)  & First Order Oracle                     \\
\hline
Mirror-Prox \cite{nemirovski2004prox}     &          $O\Big(\frac{L + \sigma_{\max}}{\epsilon}\Big)$     &  Without Acceleration               & \eqref{eq:oracle-2}            \\
\hline
Dual Extrapolation \cite{nesterov2007dual}     &              $O\Big(\frac{L + \sigma_{\max}}{\epsilon}\Big)$                &        Without Acceleration      & \eqref{eq:oracle-2}                  \\
\hline
MF-BS \cite{tseng2000modified,monteiro2011complexity}   &              $O\Big(\frac{L + \sigma_{\max}}{\epsilon}\Big)$                   &  Without Acceleration      &\eqref{eq:oracle-2}                                             \\
\hline
Primal-Dual\cite{chambolle2011first}    &   $O\Big(\frac{L + \sigma_{\max}}{\epsilon}\Big)$                      &  Without Acceleration  & \eqref{eq:oracle-2}   \\
\hline
APD \cite{chen2014optimal}   &                     $O\Big(\sqrt{\frac{L}{\epsilon}}+\frac{\sigma_{\max}}{\epsilon}\Big)$                        &   With Acceleration  & \eqref{eq:oracle-2}     \\
\hline
SAMP \cite{chen2017accelerated}    &          $O\Big(\sqrt{\frac{L}{\epsilon}}+\frac{\sigma_{\max}}{\epsilon}\Big)$                        &     With Acceleration   & \eqref{eq:oracle-2}     \\
\hline
BAXG (\textbf{This paper})              &    {$\tilde{O}\Big(\frac{\sigma_{\max}}{\sigma_{\min}}\sqrt{\frac{L}{\epsilon}} \log \frac{1}{\epsilon}\Big)$    }                            &  With Acceleration  & \eqref{eq:extend-oracle-2}  \\
\hline
\end{tabular}
\end{table*}

The third key is that when $\mA$ has full column rank,  the strongly convex minimization subproblem will have a condition number no more than $\frac{\sigma_{\max}^2}{\sigma_{\min}^2}$. As a result, to attain certain accuracy $\epsilon,$ the subproblem can be solved by the AGD method in at most $\tilde{O}(\frac{\sigma_{\max}}{\sigma_{\min}}\log \frac{1}{\epsilon})$ inner iterations, where   $\tilde{O}$ hides the logarithmic factors about $\sigma_{\min}, \sigma_{\max}$ and $L$. Finally by combining the number of iterations of the outer and inner loops, we prove that under Assumptions \ref{ass:g} and \ref{ass:A}, one can solve the bilinear convex concave minimax problem \eqref{eq:con-con-prob} in $\tilde{O}(\frac{\sigma_{\max}}{\sigma_{\min}}\sqrt{\frac{L}{{\epsilon}}} \log \frac{1}{\epsilon})$ number of accessing the extended first order oracle.

\subsection{Related Work}
Most methods for blinear convex concave minimax problems are designed for the general problem \eqref{eq:con-con}, where $\gW$ can be any closed convex set in $\bbR^d$ that admits an efficient projection operator and $\mA$ does not need to be full column rank. As a result, when applying to the problem \eqref{eq:con-con-prob}, they cannot explore the particular structure hence the rate is bounded by $O(1/\epsilon)$. In Table \ref{tb:result}, we give the related complexity results. As shown in Table \ref{tb:result}, by accessing the first order oracle \eqref{eq:oracle-2} and without Nesterov's acceleration, the convergence rate of the methods in \cite{nemirovski2004prox,nesterov2007dual,tseng2000modified,monteiro2011complexity,chambolle2011first} is $O\Big(\frac{L + \sigma_{\max}}{\epsilon}\Big)$; while by exploiting Nesterov's acceleration \cite{chen2014optimal,chen2017accelerated}, the dependence of the smoothness constant $L$ can be substantially improved, while the overall rates are still $O(1/\epsilon).$ In this paper, by accessing the extended first order oracle \eqref{eq:extend-oracle-2}, the BAXG method can explore the strong convexity induced from the full column rank assumption of $\mA.$ As a result, the rate is substantially improved to $\tilde{O}\Big(\frac{\sigma_{\max}}{\sigma_{\min}}\sqrt{\frac{L}{\epsilon}} \log \frac{1}{\epsilon}\Big)$. 

The full rank property of $\mA$ is also useful in settings where $f(\vw)$ is strongly convex but $h(\vv)$ is not strongly convex. As shown in \cite{wang2017exploiting,du2017stochastic,du2019linear}, by assuming $f(\vw)$ being strongly convex and $\mA$ having full column rank, first order primal-dual methods can have linear convergence rates. Compared with this line of research, our work does not assume strong convexity for either $f(\vw)$ or $h(\vv).$ Meanwhile, the proposed BAXG method is an accelerated extension for the extragradient method \cite{korpelevich1976extragradient}, while the algorithms proposed and studied in \cite{wang2017exploiting,du2017stochastic,du2019linear} are primal-dual methods \cite{chambolle2011first}.

\subsection{Motivating Examples}\label{eq:moti-exam}
Our main motivation to carefully examine the complexity for the class of problems in \eqref{eq:con-con-prob} is because many problems of practical importance can be reduced to this form.

The first example is the following linear equality constrained smooth optimization problem:
\begin{equation}
\min_{\vw\in\bbR^d} f(\vw) \quad \mbox{s.t.}\quad  \mA^T \vw = \b, \label{eq:smooth-linear}
\end{equation}
where $f(\vw)$ is smooth on $\bbR^d$, $\mA\in\bbR^{d\times n}$ with $n\le d$ is a matrix with full column rank. For instance, let $f(\vw)$ denote a smooth surrogate of the $\ell_1$-norm such as 
\begin{align}
R(\vw) \doteq \sum_{i=1}^d\frac{1}{a}(\log (1+\exp(a w_i) + \log (1+\exp(-a w_i)),  \label{eq:R}
\end{align}
where $a>0$\footnote{If $a$ is large, then $R(\vx)\approx \|\vx\|_1$ \cite{schmidt2007fast}}. It should be noted that $R(\vw)$ is smooth but not strongly convex. Then \eqref{eq:smooth-linear} corresponds to a smoothed version of basis pursuit \cite{chen2001atomic} for compressed sensing. The problem in \eqref{eq:smooth-linear} also arises in the subproblem of Newton method, where $f(\vw)$ denotes the quadratic approximation around a point and $\mA$ denotes the linear equality constraint. Similar to \eqref{eq:con-con-exam}, \eqref{eq:smooth-linear} is equivalent to the following minimax problem 
\begin{align}
\min_{\vw\in\bbR^d}\max_{\vv\in\bbR^n}f(\vw) + \langle \vw, \mA\vv\rangle - \vb^T\vv,    \label{eq:linear-smooth}
\end{align}
which is a particular instance of \eqref{eq:con-con-prob} and can be solved by our algorithm with the accelerated rate $O(\frac{1}{\sqrt{\epsilon}}\log \frac{1}{\epsilon})$.

In computer vision, an important model for robust face recognition is dense error correction \cite{wright2008dense}, which lends itself to solve the following problem: 
\begin{align}
\min_{\vw\in\bbR^d}\|\mA^T\vw - \vb\|_1+\lambda\|\vw\|_1, \label{eq:des-ori}
\end{align}
where $\vw\in\bbR^d, \mA\in\bbR^{d\times n}$ and $\lambda>0.$ As shown in \cite{wright2008dense}, $\mA$ can be a random matrix with $d$  larger than $n$. As a result, $\mA$ typically has full column rank. If we consider a smooth surrogate of $\|\vw\|_1$ such as the $R(\vw)$ in \eqref{eq:R}, then we have
\begin{align}
 &\min_{\vw\in\bbR^d}\|\mA^T\vw - \vb\|_1+\lambda R(\vw)\label{eq:desn-oo}\\
 =&\min_{\vw\in\bbR^d}\max_{\|\vv\|_{\infty}\le 1}  \lambda R(\vw) + \langle \vw, \mA\vv\rangle -\vb^T\vv,   \label{eq:dens}
\end{align}
which again is a particular instance of the problem \eqref{eq:con-con-prob}.

For the above nonsmooth optimization problem, the minimax reformulation in \eqref{eq:dens} is used to smooth the original nonsmooth problem \eqref{eq:desn-oo}. That is, by \emph{increasing the dimension} via dual variables $\vv\in\gV$, the nonsmooth problem is reduced to a smooth problem with a simple closed convex set $\gV$ that admits efficient projection operation. If so, the black box complexity result for the nonsmooth optimization can be significantly improved from $O(1/\epsilon^2)$ to $O(1/\epsilon)$ \cite{nesterov2005smooth}. At a high level, such a smooth technique indicates that increasing dimension helps find the underlying \emph{smooth structure} in nonsmooth problems. 

In this paper, to further reduce the complexity, the proposed BAXG algorithm essentially explores the opposite direction by \emph{decreasing the dimension}. When solving the nontrivial subproblem in each iteration, with $\vw\in\bbR^d$, we can reduce the subproblem $w.r.t.$ both $\vw$ and $\vv$ to a problem of $\vv$ only. Such a transformation helps us exploit the strong convexity of $\frac{1}{2}\vv^T\mA^T\mA\vv$ (which is implied by $\mA$ being full column rank) and allow the nontrivial  subproblem to be solved {\em in a linear rate}. As a result, when the BAXG method is used in solving nonsmooth problems such as \eqref{eq:desn-oo}, it increases the dimension of the original problem to explore the smooth structure in the nonsmooth problem, and yet decreases the dimension of the subproblem in each iteration to exploit the \emph{strongly convex structure} hidden in the bilinear interacting term in \eqref{eq:H}. Whenever such a strongly convex structure exists ($i.e.,$ Assumption \ref{ass:A}
 holds), the $O(1/\epsilon)$ rate can be significantly improved to $O(1/\sqrt{\epsilon}\log \frac{1}{\epsilon}).$ 

Notations: Let $\|\cdot\|$ denote the Euclidean norm $\|\cdot\|_2.$ For $k\in\{1,2,\ldots\}$, let $[k] \doteq \{1,2,\ldots, k\}.$

\section{Bilinear Accelerated Extragradient Method}\label{sec:axg}
In this section, we introduce the Bilinear Accelerated Extragradient (BAXG) method to solve the MVIP$(\nabla g + \mQ, \gX)$, as outlined in Algorithm \ref{alg:axg}, where we assume that the MVIP$(\nabla g + \mQ, \gX)$ satisfies Assumptions \ref{ass:g} and \ref{ass:A}. In the Step 2 of Algorithm \ref{alg:axg}, we set the values of two sequences $\{a_k\}$ and $\{A_k\}$, where $L$ is the smoothness constant of $g(\vx)$ in Assumption \ref{ass:g}. In the Step 3, we initialize $\vx_0$ and $\vz_0$ as the same value in $\gX.$ From Step 4 to 10, we perform $K$ iterations and return the last iterate $\vx_K.$ In the $K$ iterations, we generalize the approximate backward Euler discretization based acceleration methods to the MVIP$(\nabla g + \mQ, \gX)$ setting.

Following \cite{song2019towards}, we describe each iteration of Algorithm \ref{alg:axg} starting from the analysis of the ``estimation sequence'' in the Step 8. Then we show that how the other steps arise to cancel the error caused by approximate backward Euler discretization.
Before describing each iteration, we define the following linear function $\hat{f}(\vx; \vx_{i}, \hat{\vz}_i, \vy)$ $w.r.t.$ $\vx$ such that $\forall \vy\in\gX$, 
\begin{align}
\hat{f}(\vx; \vx_{i}, \hat{\vz}_i, \vy) & \doteq g(\vx_i) + \langle \nabla g(\vx_i), \vx - \vx_i\rangle + \langle \mQ\hat{\vz}_i, \vx-\vy\rangle, \label{eq:f-hat}
\end{align}
where $\hat{\vz}_i$ and $\vx_i$ are defined according to the Steps 6 and 7 of the $i$-th iteration of Algorithm \ref{alg:axg}. Compared with the linear function used in the classical estimation sequence \cite{nesterov1998introductory}, we use an extra term $\langle \mQ\hat{\vz}_i, \vx-\vy\rangle$ for the interacting term $\mQ$ in MVIP$(\nabla g + \mQ, \gX)$. Then in the Step 8, the estimation sequence $\psi_k(\vx)$ is defined as follows: $\forall k\ge 0, $ 
\begin{eqnarray}
\psi_k(\vx)\doteq\sum_{i=1}^{k} a_{i} \hat{f}(\vx; \vx_{i}, \hat{\vz}_i,\vy) +  \|\vx-\vx_0\|^2, \label{eq:dis}
\end{eqnarray}
where the sequence $\{a_i\}$ is specified in the Step 2.
Meanwhile, in \eqref{eq:dis}, when $k=0$, we let $\varphi_0(\vx) = \|\vx-\vx_0\|^2$ and thus $\vz_0=\argmin_{\vx\in\bbR^d} \|\vx-\vx_0\|^2 = \vx_0.$ From \eqref{eq:dis}, $\psi_k(\vx)$ can be recursively defined as
$$\psi_k(\vx) = \psi_{k-1}(\vx) + a_{k} \hat{f}(\vx; \vx_{k}, \hat{\vz}_k,\vy),$$  
thus the weighted sum of linear functions $\hat{f}(\vx; \vx_{i}, \hat{\vz}_i,\vy)$ 
in \eqref{eq:dis} can be computed recursively. As a result, the cost of the Step 8 of Algorithm \ref{alg:axg} is $O(d+n)$ plus the projection cost on $\gV.$

\begin{algorithm}[t!]
\caption{Bilinear Accelerated Extragradient Method}\label{alg:axg}
\begin{algorithmic}[1]
\STATE \textbf{Input: } The MVIP$(\nabla g + \mQ, \gX)$   satisfying Assumptions \ref{ass:g} and \ref{ass:A}.
\vspace{0.02in}
\STATE  $\forall k\ge 1, A_k = \frac{1}{4L}k^2, a_k = A_k - A_{k-1}$ with $A_0=0.$ 
\STATE $\vx_0 = \vz_0\in\gX$. 
\FOR{ $k = 1,2,\ldots, K$ }
\STATE $\hat{\vx}_{k-1} = \frac{A_{k-1}}{A_k}\vx_{k-1} + \frac{a_k}{A_k} \vz_{k-1}.$
\STATE Find a $\hat{\vz}_k\in \gX$ by Algorithm \ref{alg:pagd} such that $\forall \vz \in \gX,$
\begin{small}
\begin{align}
&\!\!\!\!\!\!a_{k}\Big\langle  \nabla {g}(  \hat{\vx}_{k-1}) + \mQ \vz_{k-1} +  \Big(\mQ+  \frac{2}{a_k}\mI \Big)(\hat{\vz}_k - \vz_{k-1}) , %
\hat{\vz}_k - \vz\Big\rangle 
-\frac{1}{2}\Big( \| \hat{\vz}_k - \vz_{k-1} \|^2 +  \|\hat{\vz}_k - \vz\|^2\Big)\le 0.\label{eq:cond}
\end{align}
\end{small}
\STATE $\vx_k = \hat{\vx}_{k-1} + \frac{a_k}{A_k}(\hat{\vz}_k - \vz_{k-1}).$
\STATE  $\vz_{k} =  \argmin_{\vx\in \gX}\psi_k(\vx)$ with $\psi_k(\vx)$ in \eqref{eq:dis}.
\ENDFOR
\STATE \textbf{return } $\vx_K$.
\end{algorithmic}	
\end{algorithm}

Based on the definition of the estimation sequence $\psi_k(\vx)$ in \eqref{eq:dis} and the optimality condition of $\vz_k$, we can give an upper bound for $\psi_k(\vz_k)$ as in Lemma \ref{lem:dis-upper}.
\begin{lemma}\label{lem:dis-upper}
 $\forall k\ge 0$ and $\forall \vy\in\gX$, one has $\psi_k(\vz_k)\le A_k g(\vy) + \|\vy-\vx_0\|^2$. 
\end{lemma}
\begin{proof}
See Section \ref{subsec:lem:dis-upper}.
\end{proof}

In addition, we can establish a lower bound for $\psi_k(\vz_k)$ too as below.  
\begin{lemma}\label{lem:dis-lower}
$\forall i\in [k]$, let $$E_{i}\doteq a_{i}\left\langle \nabla g(\vx_{i}) + \mQ\hat{\vz}_i,  \hat{\vz}_i-\vz_i\right\rangle -   \|\vz_{i}-\vz_{i-1}\|^2.
$$
Then  $\forall k\ge 1$ and $\forall \vy\in\gX$, one has 
\begin{align}
&A_{k}(g(\vx_k)+\langle \mQ\vy, \vx_k-\vy\rangle) \le \psi_{k}(\vz_{k}) + \sum_{i=1}^{k} E_{i}.\label{eq:dis-lower}
\end{align}
\end{lemma}
\begin{proof}
See Section \ref{subsec:lem:dis-lower}.
\end{proof}

In Lemma \ref{lem:dis-lower}, $\{E_i\}$ can be viewed as a sequence of ``error'' terms that we hope to be less than $0$. If we set $\hat{\vz}_i = \vz_i$ which is called \emph{backward Euler discretization}, then $E_{i}\le 0$. However, in this case the resulted problem in the Step 8 will become a fixed point problem and is as difficult as the original problem. To address this difficulty, we can consider \emph{approximate backward Euler discretization} by approximating $\vz_i$ with a solution of an easier subproblem. To attain this goal, $\forall i\in [k],$ we consider a convex combination of $\vx_{i-1}$ and $\vz_{i-1}$ as $\hat{\vx}_{i-1}$ in the Step 5 of the $i$ iteration. Then we approximate $\vz_i$ by $\hat{\vz}_i$ as a solution to a subproblem in the Step 6, which is to be solved by Algorithm \ref{alg:pagd} in Section \ref{sec:subprob}. Based on the careful choice of $\hat{\vz}_i$ and $\vx_i$, we have Lemma \ref{lem:E-3}. 

\begin{lemma}\label{lem:E-3}
 $\forall  i\in [k]$, one has
\begin{align}
E_i\le \frac{1}{2}\Big(\Big(\frac{a_i^2L}{A_i}\Big)^2 -1\Big)\|\hat{\vz}_i - \vz_{i-1}\|^2. \label{eq:E-5}
\end{align}
\end{lemma}
\begin{proof}
Sec Section \ref{sec:lem:E-3}.
\end{proof}
Then by Lemma \ref{lem:E-3}, $\forall i\in[k],$ $E_i\le 0$ if the sequences $\{a_i\}, \{A_i\}$ are set according to the Step 2. Then combining Lemmas \ref{lem:dis-upper}, \ref{lem:dis-lower} and \ref{lem:E-3}, we have Theorem \ref{thm:axg}.

\begin{theorem}[Iteration Complexity of Algorithm \ref{alg:axg}]\label{thm:axg}
In all the $k$-th iterations, we have $\forall \vy \in \gX,$ 
\begin{eqnarray}
g(\vx_k)\! -\! g(\vy)+\langle \mQ\vy, \vx_k\!-\!\vy\rangle \le \frac {4L\|\vy - \vx_0\|^2}{k^2}.
\end{eqnarray}
\end{theorem}
\begin{proof}
By the definition of $A_k, a_k$ in the Step 2 of Algorithm \ref{alg:axg}, we have $\forall k\ge 1,$ 
\begin{align}
\frac{a_k^2L}{A_k} \le 1. \label{thm:axg0}
\end{align}
So by combining Lemmas \ref{lem:dis-upper}, \ref{lem:dis-lower} and \eqref{thm:axg0}, we have $\forall k\ge 1$ and $\forall \vy\in\gX,$  
\begin{align}
&A_{k}(g(\vx_k)+\langle \mQ\vy, \vx_k-\vy\rangle) \nonumber\\
\le& \psi_{k}(\vz_{k}) + \sum_{i=1}^{k} E_{i} \; \le\; A_k g(\vy) + \|\vy-\vx_0\|^2  + \sum_{i=1}^{k} E_{i}\nonumber\\ 
\le& A_k g(\vy) + \|\vy\!-\!\vx_0\|^2 +  \frac{1}{2}\sum_{i=1}^{k}\Big(\Big(\frac{a_i^2L}{A_i}\Big)^2\!\!-1\Big)\|\hat{\vz}_i \!-\! \vz_{i-1}\|^2 \nonumber\\ 
\le& A_k g(\vy) + \|\vy-\vx_0\|^2. \label{eq:thm:axg-1}
\end{align}
Then by the definition of $A_k = \frac{k^2}{4L}$ and after a simple arrangement,  Theorem \ref{thm:axg} is proved.
\end{proof}

By Theorem \ref{thm:axg}, to find an $\epsilon$-accurate solution satisfying \eqref{eq:appro-weak}, we only need $O(\frac{1}{\sqrt{\epsilon}})$ number of (outer) iterations. The remaining issue is to show the subproblem \eqref{eq:cond} can be solved in linear rate $O(\log \frac{1}{\epsilon})$.

\section{Solving the Subproblem}\label{sec:subprob}
In this section, we show how to efficiently solve the subproblem \eqref{eq:cond} of Step 6 in the $k$-th iteration of Algorithm \ref{alg:axg}. 
To simplify and differentiate from the notion used in the main algorithm, we define 
\vspace{-2mm}
\begin{align}
c\doteq   a_k = (2k-1)/L \label{eq:sim-notation2}  
\vspace{-1mm}
\end{align}
and use the following notation for the subproblem (see the original definition of $\vx$ in \eqref{eq:H})
\begin{align}
\left[\begin{matrix}
     \vw^0 \\
     \vv^0 
\end{matrix}\right]\doteq\nabla {g}(\hat{\vx}_{k-1}) + \mQ \vz_{k-1}, \;
\left[\begin{matrix}
     \vw^1 \\
     \vv^1 
\end{matrix}\right]\doteq& \vz_{k-1},\;
\left[\begin{matrix}
     \tilde{\vw} \\
     \tilde{\vv} 
\end{matrix}\right]\doteq\hat{\vz}_k.
\label{eq:sim-notation}
\end{align}
With this notation, all the $\vw^0,\vv^0, \vw^1$ and $\vv^1$ are constants as far as the subproblem is concerned. $\tilde{\vw}\in\bbR^d,\tilde{\vv}\in\gV$ are the solutions we hope to find for the subproblem. To find $\tilde{\vw}$ and $\tilde{\vv}$ efficiently, a key observation is that due to $\tilde{\vw}\in\bbR^d,$ it is enough to express $\tilde{\vw}$ by a closed form $w.r.t.$ $\tilde{\vv}$ and then solve a {\em strongly convex} minimization problem to find $\tilde{\vv}$. 
We state this fact in Lemma \ref{lem:refor}.
\begin{lemma}\label{lem:refor}
Let $\{c, \vw^0,\vv^0,\vw^1,\vv^1, \tilde{\vw}, \tilde{\vv}\}$ be defined in \eqref{eq:sim-notation2} and \eqref{eq:sim-notation}. Then by setting 
\begin{align}
\tilde{\vw} = \vw^1 - \frac{c}{2} (\vw^0 + \mA(\tilde{\vv}-\vv^1)),  \label{eq:tilde-w-v} 
\end{align}
then  the Step 6 of the $k$-th iteration of Algorithm \ref{alg:axg} is reduced to finding a solution $\tilde{\vv}\in\gV$ such that $\forall \vv \in \gV$,    
\begin{align}
&\big\langle \nabla l(\tilde{\vv}), \tilde{\vv}-  {\vv}\big\rangle \le \frac{2}{c^2}\|\tilde{\vv} - \vv^1\|\| \tilde{\vv} - \vv\|, \label{eq:cond-refor3}   
\end{align}
where $l(\vv)$ is a quadratic function defined by: $\forall \vv\in\gV,$
\begin{align}
l(\vv) \;\doteq\; & \frac{1}{2}(\vv-\vv^1)^T\Big(\mA^T\mA + \frac{4}{c^2}\mI\Big)(\vv-\vv^1)+ \Big( \frac{2}{c}\vv^0 + \mA^T\vw^0\Big)^T (\vv-\vv^1). \label{eq:cond-refor4}
\end{align}
\end{lemma}
\begin{proof}
See Section \ref{sec:lem:refor}
\end{proof}
In Lemma \ref{lem:refor}, $l(\vv)$ is a quadratic function. By Assumption \ref{ass:A},  $l(\vv)$ is smooth with the smoothness constant $L_l$ and strongly convex with strong convexity parameter $\sigma_l$, where $L_l$ and $\sigma_l$
are defined as 
\begin{align}
L_l\doteq\sigma_{\max}^2+\frac{4}{c^2},\quad \sigma_l \doteq\sigma_{\min}^2+\frac{4}{c^2}.  \label{eq:L-sigma-l}
\end{align}

Let $\vv^*$ is the optimal solution of $\min_{\vv\in\gV}l(\vv).$ Then by the first order optimality condition\footnote{From the perspective of variational inequality \cite{facchinei2007finite}, the $\vv^*$ satisfying \eqref{eq:strong-v} is called the strong solution of the related variational inequality problem about $(\nabla l(\vv), \gV)$.}
 of $\vv^*$, we have: $\forall \vv\in\gV$, 
\begin{align}
\langle \nabla l(\vv^*), \vv^*-\vv\rangle \le 0. \label{eq:strong-v}
\end{align}
Comparing with \eqref{eq:strong-v}, the condition in \eqref{eq:cond-refor3} demands us to find kind of an $\epsilon$-accurate solution in terms of the first order optimality condition
with 
\begin{align}
\epsilon \doteq \frac{2}{c^2}\|\tilde{\vv} - \vv^1\|\| \tilde{\vv} - \vv\|. \label{eq:l-eps}
\end{align}

To find a $\tilde{\vv}$ satisfying \eqref{eq:cond-refor3}, in Algorithm \ref{alg:pagd}, we consider a projected accelerated gradient descent (PAGD) method. The PAGD method is a minor variant of the 
AGD \cite{nesterov1998introductory} for constrained strongly convex problems, where we assume that an efficient projection operator onto $\gV$ exists. Compared with common convex minimization methods which aim to find an $\epsilon$-solution $\vv$ in terms of the gap of objective function $l(\vv) - l(\vv^*)\le \epsilon$, the PAGD method aims to find an $\epsilon$-accurate solution in terms of the first order optimality condition in \eqref{eq:strong-v} with $\epsilon$ defined in \eqref{eq:l-eps}. This task is nontrivial because the upper bound for the gap of objective function can not be directly used, despite the convex property $l(\tilde{\vv}) - l(\vv)\le \langle \nabla l(\tilde{\vv}), \tilde{\vv} - \vv\rangle, \forall \vv\in\gV.$ To this end, in Algorithm \ref{alg:pagd}, instead of bounding the gap of objective function, we bound the residual norm $\|\vv_t-\hat{\vv}_{t-1}\|$, which in turn can be used to obtain a solution $\tilde{\vv}$ satisfying \eqref{eq:cond-refor3} (See Lemma \ref{lem:strong-dist}). 

Algorithm \ref{alg:pagd} is particularly designed for the problem of finding $\tilde{\vv}$ satisfying \eqref{eq:cond-refor3}. As a result, we initialize $\vv_0$ and $\vu_0$ as the $\vv^1$ defined in \eqref{eq:sim-notation} in the Step 3 of Algorithm \ref{alg:pagd}. Meanwhile, in the Step 9, we set the best $\vv_t$ as $\tilde{\vv}$, which has the minimal residual norm $\|\vv_t - \hat{\vv}_{t-1}\|$ among all iterates.  Then in the Step 10, we obtain $\tilde{\vw}$ from $\tilde{\vv}$ according to \eqref{eq:tilde-w-v}. Finally, according to \eqref{eq:sim-notation}, $\tilde{\vw}$ and $\tilde{\vv}$ are stacked as the $\hat{\vz}_k$, which is the desired solution for the Step 6 of the $k$-th iteration of Algorithm \ref{alg:axg}. 

All the other steps are standard steps for the well-known projected accelerated gradient descent method \cite{diakonikolas2019approximate}. In the Step 2, we set the sequences $\{b_t\}, \{B_t\}$. Then in the $t$-th iteration, we perform a convex combination of $\vv_{t-1}$ and $\vu_{t-1}$ in the Step 5, a gradient descent step to obtain $\vv_t$ in the Step 6, and a dual averaging step in the Step 7. (For a detailed explanation for these steps of accelerated first order methods, the reader may refer to \cite{diakonikolas2019approximate}).

In the Step 7 of Algorithm \ref{alg:pagd}, we define $\vu_t$ as the minimizer in $\gV$ of the following estimation sequence: 
\begin{align}
\varphi_t(\vu) \doteq &\sum_{i=1}^t b_i \big(l(\hat{\vv}_{i-1}) + \langle \nabla l(\hat{\vv}_{i-1}), \vu-\hat{\vv}_{i-1}\rangle 
+ \frac{\sigma_l}{2}\|\vu-\hat{\vv}_{i-1}\|^2\big) + \frac{1}{2}\|\vu - \vv_0\|^2, \label{eq:varphi}  
\end{align}
where $\{\hat{\vv}_{i-1}\}$ are defined in the Step 5 of the $t$ iterations. Again, $\varphi_t(\vv)$ can be recursively defined as
\begin{align}
\varphi_t(\vu) \doteq \; & \varphi_{t-1}(\vu) + b_t \big(l(\hat{\vv}_{t-1}) + \langle \nabla l(\hat{\vv}_{t-1}), \vu-\hat{\vv}_{t-1}\rangle + \frac{\sigma_l}{2}\|\vu-\hat{\vv}_{t-1}\|^2\big).
\end{align}

\begin{algorithm}[t!]
\caption{Projected Accelerated Gradient Descent}\label{alg:pagd}
\begin{algorithmic}[1]
\STATE \textbf{Input: } ,  $\vw^0, \vv^0, \vw^1, \vv^1$ defined in \eqref{eq:sim-notation}, the quadratic function $l(\vv)$ with $\vv \in\gV$ defined in \eqref{eq:cond-refor4}, $L_l, \sigma_l$ in \eqref{eq:L-sigma-l}. 
\vspace{0.02in}
\STATE $\{b_t\}, \{B_t\}$ are sequences satisfying $B_0 = 0$  and 
\begin{small}
$$
\!\!\!\!\!  B_1 = \frac{1}{2L_l},
 B_t = B_1\Big(1+\sqrt{\frac{\sigma_l}{2L_l}}\Big)^{t-1}, b_t = B_t - B_{t-1}.
$$
\end{small}
\STATE $\vv_0 = \vu_0 = \vv^1$. 
\FOR{$t = 1, 2, \ldots, T$}
\STATE $\hat{\vv}_{t-1} = \frac{B_{t-1}}{B_t}\vv_{t-1} + \frac{b_t}{B_t} \vu_{t-1}$.
\STATE $\vv_t \!=\! \argmin_{\vv\in\gV}\left\{\langle \nabla l(\hat{\vv}_{t-1}), \vv\rangle + L_l\|\vv-\hat{\vv}_{t-1}\|^2\right\}$.
\STATE  $\vu_{t} = \argmin_{\vu\in\gV}\varphi_t(\vu)$ with $\varphi_t(\vu)$ in  \eqref{eq:varphi}. 
\ENDFOR
\STATE For so computed set $\left\{ (\vv_t, \hat{\vv}_{t-1}) \right\}_{t\in [T]}$, select: $$\tilde{\vv} = \argmin_{\vv_t: t\in [T]}\|\vv_t - \hat{\vv}_{t-1}\|.$$
\STATE Let $\tilde{\vw} = \vw^1 - \frac{c}{2} (\vw^0 + \mA(\tilde{\vv}-\vv^1)).$
\STATE \textbf{return } $\hat{\vz}_k = \left[\begin{matrix}
     \tilde{\vw}\\
     \tilde{\vv} 
\end{matrix}\right]$.
\end{algorithmic}	
\end{algorithm}

We prove our result by performing estimation sequence analysis. First, an upper bound of $\varphi_t(\vu_t)$ is given below. 
\begin{lemma}\label{lem:varphi-upper}
$\forall t\ge 0$, we have 
\begin{eqnarray}
\varphi_t(\vu_t)&\le&  B_t l(\vv^*)+ \frac{1}{2}\|\vv^*-\vv_0\|^2.
\end{eqnarray}
\end{lemma}
\begin{proof}
See Section \ref{sec:lem:varphi-upper}. 
\end{proof}

Then we give a lower bound of  $\varphi_t(\vu_t)$ in Lemma \ref{lem:varphi-lower}.
\begin{lemma}\label{lem:varphi-lower}
$\forall t\ge 0$, we have 
\begin{eqnarray}
B_t l(\vv_t) + \sum_{i=1}^t \frac{L_l B_i}{2} \|{\vv}_i - \hat{\vv}_{i-1}\|^2 &\le& \varphi_t(\vu_t). 
\end{eqnarray}
\end{lemma}
\begin{proof}
See Section \ref{sec:lem:varphi-lower}. 
\end{proof}

By Lemmas \ref{lem:dis-upper}, \ref{lem:dis-lower} and the fact $l(\vv_t)\ge l(\vv^*),$ we can give an upper bound for $\min_{i\in[t]}\|\vv_i - \hat{\vv}_{i-1}\|.$ To provide a bound for $\langle \nabla l(\vv_i), \vv_i - \vv\rangle$, according to the optimal condition of $\vv_i$ in the Step 6 of the $i$-th iteration of Algorithm \ref{alg:pagd} and the optimal condition of $\vv^*$ for the problem $\min_{\vv\in\gV}l(\vv),$ we have Lemma \ref{lem:strong-dist}.
\begin{lemma}\label{lem:strong-dist}
$\forall i\ge 1, $ we have $\vv \in \gV,$
\vspace{-1mm}
\begin{eqnarray}
\langle \nabla l(\vv_i), \vv_i - \vv\rangle &\le& 3L_l\|\vv_i - \hat{\vv}_{i-1}\|\|\vv_i - \vv\|,\nonumber\\
\|\vv_i-\vv^*\| &\le& \frac{3 L_l}{\sigma_l}\|\vv_i - \hat{\vv}_{i-1}\|.
\vspace{-2mm}
\end{eqnarray}
\end{lemma}
\begin{proof}
See Section \ref{sec:lem:strong-dist}. 
\end{proof}

Combining Lemmas \ref{lem:varphi-upper}, \ref{lem:varphi-lower} and \ref{lem:strong-dist}, and using the definition of $L_l, \sigma_l$ in \eqref{eq:L-sigma-l}, $c$ in \eqref{eq:sim-notation2} and $\hat{\vz}_k$ in \eqref{eq:sim-notation}, we have Theorem \ref{thm:subprob}.   
\begin{theorem}[Iteration Complexity of Algorithm \ref{alg:pagd}]\label{thm:subprob}
With at most
\vspace{-2mm}
\begin{align}
t\ge \frac{\log\Big(1\!+\!9\sqrt{2}L_l^{\frac{3}{2}}\sigma_l^{-\frac{3}{2}}(\sigma_l c^2 \!+\! 1)^2  \Big)}{\log\Big(1+\sqrt{\frac{\sigma_l}{{2}L_l}}\Big)} 
= \tilde{O}\Big(\frac{\sigma_{\max}}{\sigma_{\min}}\log k \Big),\nonumber
\vspace{-2mm}
\end{align}
iterations, where $k$ denotes the $k$-th iteration of Algorithm \ref{alg:axg}, and $\tilde{O}$ hides some log factors about $\sigma_{\min}, \sigma_{\max}$ and $L$,  Algorithm \ref{alg:pagd} returns $\hat{\vz}_k$ such that the condition \eqref{eq:cond} holds.
\end{theorem}
\begin{proof}
See Section \ref{sec:thm:subprob}. 
\end{proof}

From Theorem \ref{thm:axg} and Theorem \ref{thm:subprob}, we have Corollary \ref{coro:axg2}, which is the main result of this paper.
\begin{corollary}\label{coro:axg2}
With Algorithm \ref{alg:pagd} as the subsolver, Algorithm \ref{alg:axg} will return an $\epsilon$-accurate solution that satisfies \eqref{eq:appro-weak} with at most $\tilde{O}\Big(\frac{\sigma_{\max}}{\sigma_{\min}}\sqrt{\frac{L}{\epsilon}} \log \frac{1}{\epsilon}\Big)$ number of accessing the extended first order oracle in \eqref{eq:extend-oracle-2}, where $\tilde{O}$ hides certain logarithmic factors about $\sigma_{\min}, \sigma_{\max}$ and $L.$
\end{corollary}
\begin{proof}
Simply combine the results of Theorem \ref{thm:axg} and Theorem \ref{thm:subprob}, with the extended oracle \eqref{eq:extend-oracle-2}. 
\end{proof}

\begin{remark}
Compared with the BAXG method which uses the approximate backward Euler discretization, the PAGD method in Algorithm \ref{alg:pagd} uses the forward Euler discretization. Thus, in each iteration of PAGD, only one gradient $\nabla l(\hat{\vv}_{t-1})$ needs to be computed. 
\end{remark}

\section{Problems with Linear System Structure}
In our discussion above, our convergence results are given for the general problem class in \eqref{eq:con-con-prob}. In practice,  many problems in this class may have additional  structures, which may enable us to obtain even better convergence guarantees. One subclass of  the problem class  \eqref{eq:con-con-prob} is that we restrict \eqref{eq:con-con-prob}  with $\gV = \bbR^n$. As we have seen in Section \ref{eq:moti-exam}, one example is the linear equality constrained smooth optimization problem in \eqref{eq:smooth-linear} . Another example is the saddle point formulation of the empirical mean squared projected Bellman error in reinforcement learning \cite{du2017stochastic}. When $\gV = \bbR^n$, the associated subproblem in Section \ref{sec:subprob}  will be a strongly convex quadratic minimization problem on $\bbR^n$, which in turns is a problem of solving a linear system as follows 
\begin{align}
\Big(\mA^T\mA + \frac{4}{c^2}\mI\Big)\vs = -\Big( \frac{2}{c}\vv^0 + \mA^T\vw^0\Big), \label{eq:linear-sys}
\vspace{-3mm}
\end{align}
where $\vs = \vv - \vv^1$ and  the other notations are the same as in Section \ref{sec:subprob}. 

Besides convex optimization tools such as the AGD method, we can also use the Krylov subspace methods that are designed specifically for solving a linear system, by the well-known conjugate gradient (CG) descent \cite{saad2003iterative}. Theoretically CG has the same  worst time complexity with AGD, but it can be much faster than its worst case in practice, due to the so called ``superlinear convergence behavior'' \cite{beckermann2001superlinear}. 

Meanwhile, in all the subproblems of the Step 6 in Algorithm \ref{alg:axg}, the corresponding linear system problems in \eqref{eq:linear-sys} differ by only a scaled identity matrix and the right hand side. This suggests that if it is not so expensive,
we can compute an eigenvalue decomposition of $\mA^T\mA$ beforehand with  typically $O(n^3)$ cost. Then all the subproblems of the Step 6 in Algorithm \ref{alg:axg} can be solved with cost $O(n^2)$. That is, if the number of iterations of Algorithm \ref{alg:axg} is $\Omega(n),$ then the amortized complexity of solving each subproblem exactly is $O(n^2 + nd) = O(nd)$ due to  $n\le d$ in our setting -- the same cost as accessing the oracle \eqref{eq:oracle-2}. As a result, for this subclass of problem \eqref{eq:con-con-prob} with $\gV =\bbR^n$, we can essentially find an $\epsilon$-accurate solution satisfying \eqref{eq:appro-weak} with (equivalent) $O(1/\sqrt{\epsilon})$ number of accessing the extended first order oracle in \eqref{eq:oracle-2}. 

When $n$ is large, it will be expensive to pre-compute the eigenvalue decomposition of $\mA^T\mA$. In this case, one can perform partial eigenvalue decomposition of $\mA^T\mA$ along the iterations of Algorithm \ref{alg:axg} by the Lanczos method for linear systems only with different scalar shifts and right hand sides \cite{saad1987lanczos, soodhalter2014krylov}. This normally leads to complexity results not worse than  pre-computing the eigenvalue decomposition of $\mA^T\mA$.

\section{Experiments}\label{sec:experiments}
In this section, we provide some preliminary experiments to verify the effectiveness of the proposed BAXG algorithm. We test the performance of BAXG with Algorihm \ref{alg:pagd} as the subsolver on the two motivating examples of Section \ref{eq:moti-exam}: the linear equality constrained smoothed-$\ell_1$ norm optimization problem in \eqref{eq:smooth-linear} and \eqref{eq:R}, and the dense error correction problem in \eqref{eq:des-ori}. For comparison, we also implement the extragradient (EG) method \cite{korpelevich1976extragradient}, the stochastic accelerated mirror prox (SAMP) method \cite{chen2017accelerated} for the two problems respectively. Both the EG and SAMP methods are assumed to access the first order oracle\footnote{We assume the SAMP method can access exact gradients.} in \eqref{eq:oracle-2}, while the BAXG method is assumed to access the extended first order oracle in \eqref{eq:extend-oracle-2}.

\subsection{Linear Equality Constrained Smoothed $\ell_1$-Norm Minimization}\label{sub:smooth-l1}
Using Lagrange formulation, we have 
\begin{align}
&\min_{\vw \in\bbR^d } R(\vw) \quad s.t. \quad  \mA^T\vw = \vb \nonumber\\
=&\min_{\vw\in\bbR^d} \max_{\v\in \bbR^n} \lambda R(\vw) + \langle \vw,\mA\vv\rangle - \vb^T\vv,  \label{eq:LC} 
\vspace{-1mm}
\end{align}
where 
$
R(\vw) \doteq \sum_{i=1}^d\frac{1}{a}(\log (1+\exp(a w_i) + \log (1+\exp(-a w_i))$
with $a=10^6$, $\lambda>0$ is the parameter to balance the objective function $R(\vw)$ and the constraint $\mA^T\vw = \vb.$  
In our experiments, we set 
 $\mA\in\bbR^{d \times n}$ with $d = 1,000$ and $n=500$ as a random Gaussian matrix with \emph{i.i.d.} elements from $\mathcal{N}(0,1/\sqrt{n}),$ $\vb = \mA^T\vw^*$ with $\vw^*$ is a sparse vector with random nonzero elements from $\{1,-1\}$, and $\lambda\in\{10^{-6}, 10^{-4}, 10^{-2}\}.$
Although we can estimate the global Lipschitz constant of $\lambda R(\vw)$, it is too pessimistic for practical use\footnote{With $a=10^6$ and $\lambda>0$, the global Lipschitz constant of $\lambda R(\vw)$ will be $10^6\lambda$. \cite{schmidt2007fast}}. In our experiments, we search the best Lipschitz constant of $\lambda R(\vw)$  in $\{10^{-5}, 10^{-4}, 10^{-3}, 10^{-2}, 10^{-1}, 1, 10, 10^2, 10^3\}$.

As both oracle models in \eqref{eq:oracle-2} and \eqref{eq:extend-oracle-2} have similar computational cost, we use the number of accessing (extended) first order oracle as the $x$-axis, while we use the value $\lambda|R(\vw) - R(\vw^*)| + \|\mA^T\vw - b\|$ to measure the progress of these algorithms.
In Figure \ref{fig:1}, we show the comparision of these algorithms for the three settings $\lambda \in \{10^{-6}, 10^{-4}, 10^{-2}\}.$ For the two cases $\lambda\in\{10^{-6}, 10^{-4}\}$ where the Lipschitz constant of $\lambda R(\vw)$ is not so large, because of the accelerated rate $\tilde{O}(\frac{\sigma_{\max}}{\sigma_{\min}} \sqrt{L/{\epsilon}}\log \frac{1}{\epsilon})$, the BAXG method will overtake the EG and SAMP methods after proper iterations. On the other hand, when $\lambda = 10^{-2}$, the large Lipschitz constant of $\lambda R(\vw)$ in the complexity bound of BAXG makes BAXG no longer faster than SAMP\footnote{As shown in Table \ref{tb:result}, SAMP reduces the effect of large Lipschitz constant  by acceleration.}, while still has an edge over the EG method.  

\begin{figure}
 \subfigure[$\lambda=10^{-6}$]{
\begin{minipage}[t]{0.3\linewidth}
\centering
\includegraphics[scale=0.33]{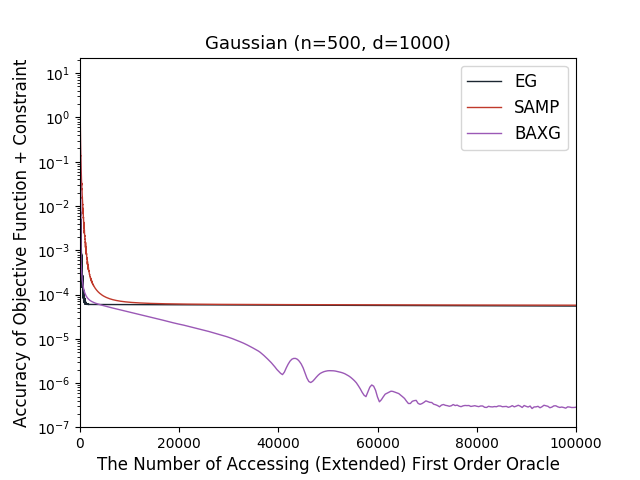}
\label{sub:fig1}
\end{minipage}%
}\hspace{2mm}
\subfigure[$\lambda=10^{-4}$]{
\begin{minipage}[t]{0.3\linewidth}
\centering
\includegraphics[scale=0.33]{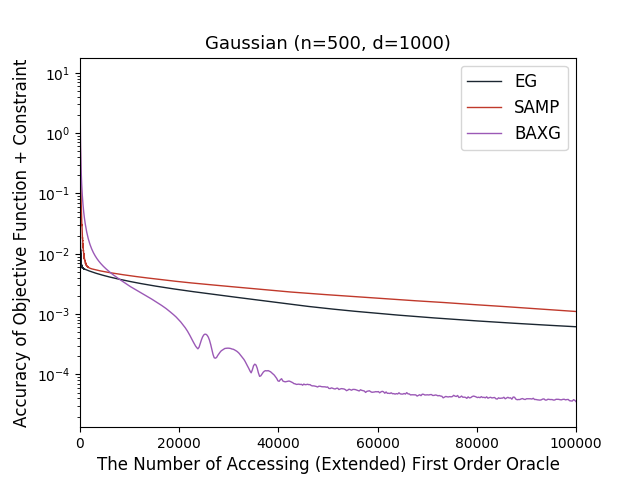}
\label{sub:fig2}
\end{minipage}%
}\hspace{2mm}
\subfigure[$\lambda=10^{-2}$]{
\begin{minipage}[t]{0.3\linewidth}
\centering
\includegraphics[scale=0.33]{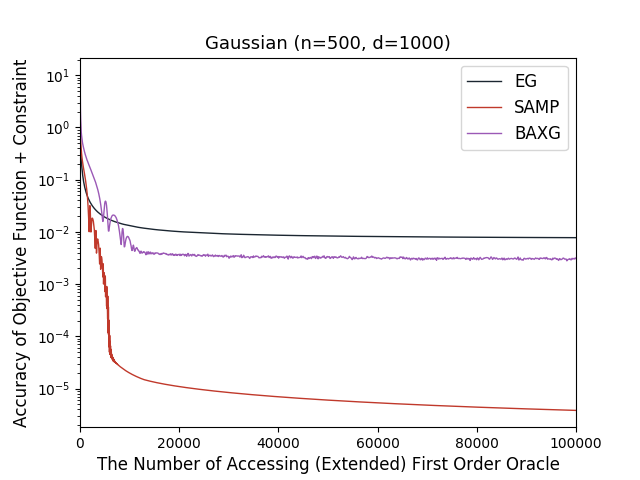}
\label{sub:fig3}
\end{minipage}
}%
\caption{Comparison of the EG, SAMP, and BAXG algorithms for linear equality constrained smoothed $\ell_1$-norm minimization.\vspace{-3mm}}\label{fig:1}
\end{figure}

\subsection{Dense Error Correction}\label{sec:add-exp}
In Section \ref{sub:smooth-l1}, we have conducted experiments on the first motivating example of linear equality constrained problem \eqref{eq:linear-smooth}  in Section \ref{eq:moti-exam}. In this section, we consider the second motivating example of dense error correction \eqref{eq:dens} and reformulate it as follows 
\begin{eqnarray}
\min_{\vw\in\bbR^d}\lambda R(\vw) + \|\mA^T\vw - \vb\|_1 = \min_{\vw\in\bbR^d}\max_{\vv:\|\vv\|_{\infty}\le 1}\lambda R(\vw) + \langle \vw, \mA\vv\rangle - \vb^T\vv, \label{eq:lad}
\end{eqnarray}
where $R(\vw)$ is the smoothed version \eqref{eq:R} of $\ell_1$ norm with $a=10^6.$ The problem \eqref{eq:lad}  is also called least absolute deviation (LAD) in regression literature.
Compared with \eqref{eq:LC}, the only difference of the minimax reformulation is that the dual variable $\vv$ is constrained in a unit $\ell_{\infty}$ ball. 

For the problem \eqref{eq:lad}, in the experiments, we initialize $\mA\in\bbR^{d\times n}$ with $d=1,000$ and $n=500$ as a random Gaussian matrix with \emph{i.i.d.} elements from $\mathcal{N}(0, 1/\sqrt{n}),$ $\vb$ is set as a random vector of values from $\{1, -1\}$ and $\lambda\in\{10^{-6}, 10^{-4}, 10^{-2}\}$. 
The best Lipschitz constant of $\lambda R(\vw)$  in $\{10^{-5}, 10^{-4}, 10^{-3}, 10^{-2},$ $10^{-1}, 1, 10, 10^2, 10^3\}$.

Again we use the number of accessing the (extended) first order oracle as the $x$-axis. Meanwhile we use the gap of objective function $(\lambda R(\vw) + \|\mA^T\vw - \vb\|_1) - (\lambda R(\vw^*) + \|\mA^T\vw^* - \vb\|_1)$ to measure the progress of the algorithms, where $\vw^*$ is found by running the BAXG method with enough iterations. 
In Figure \ref{fig:2}, we show the comparison of these algorithms for solving \eqref{eq:lad}. As we see, for both settings of $\lambda=10^{-6}$ and $\lambda=10^{-4}$, the proposed BAXG method will outperform the EG and BAXG methods, due to the accelerated rate $\tilde{O}(\frac{\sigma_{\max}}{\sigma_{\min}} \sqrt{L/{\epsilon}}\log \frac{1}{\epsilon})$. For the setting $\lambda = 10^{-2}$, shown in  Figure \ref{fig:2}(c), because of the large Lipschitz constant of $\lambda R(\vw)$, the constant in the complexity bound of BAXG will be very large such that in the low accurate region, the performance of BAXG will not be better than the EG and SAMP methods. Nevertheless, as shown in Figure \ref{fig:2}(c), the performance of BAXG will eventually approach that of the EG method. 

The experiments in the above and those given in Section \ref{sub:smooth-l1} on the two optimization problems all conform with the theoretical characterization of our algorithm, its accelerated convergence rate under the prescribed working conditions. 

\begin{figure}[t!]
 \centering
 \subfigure[$\lambda=10^{-6}$]{
\begin{minipage}[t]{0.3\linewidth}
\centering
\includegraphics[scale = 0.33]{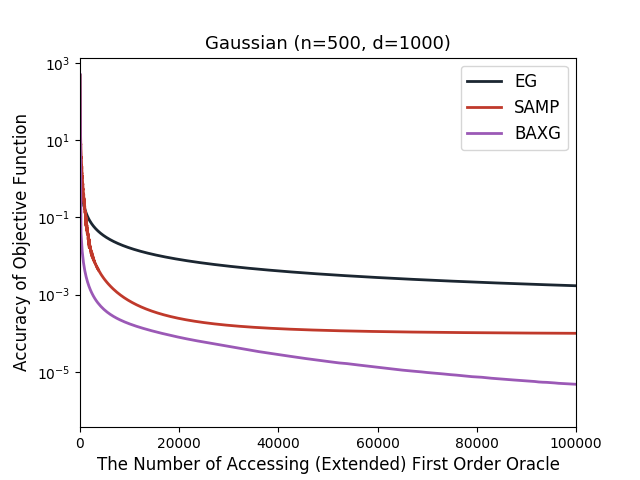}
\label{sub:fig4}
\end{minipage}%
}\hspace{2mm}
\subfigure[$\lambda=10^{-4}$]{
\begin{minipage}[t]{0.3\linewidth}
\centering
\includegraphics[scale = 0.33]{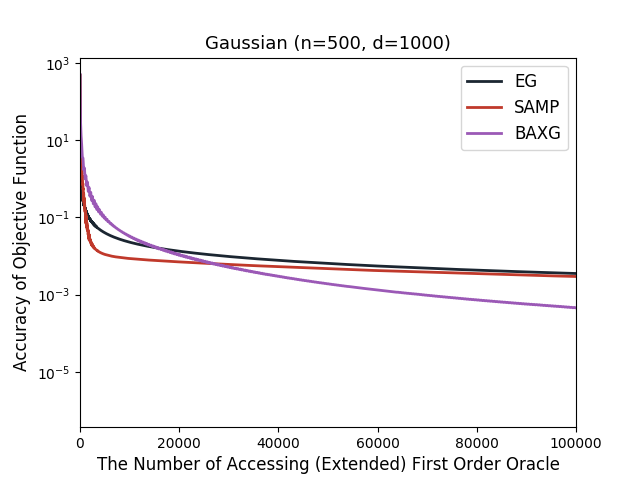}
\label{sub:fig5}
\end{minipage}%
}\hspace{2mm}
\subfigure[$\lambda=10^{-2}$]{
\begin{minipage}[t]{0.3\linewidth}
\centering
\includegraphics[scale = 0.33]{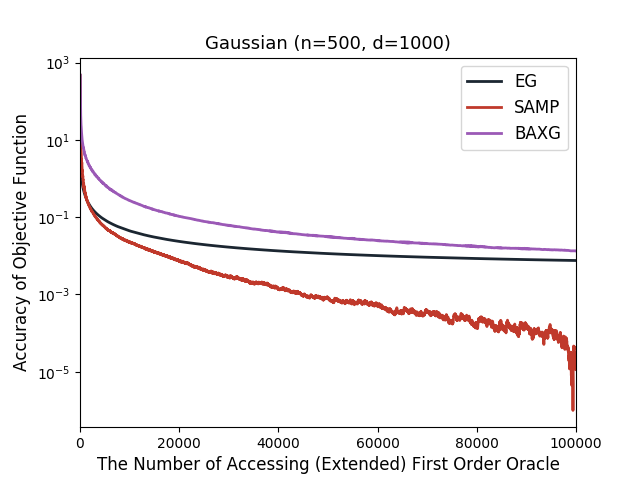}
\label{sub:fig6}
\end{minipage}
}%
\caption{Comparision of the EG, SAMP, and BAXG algorithms for dense error correction.}\label{fig:2}
\end{figure}

\section{Conclusion and Future Work}
In this paper, we have studied an important class of bilinear convex concave minimax optimization problems $\min_{\vw\in\in\bbR^d}\max_{\vv\in\gV}f(\vw) +\langle \vw, \mA\vv\rangle - h(\vv)$, where $\gV$ is a closed convex set, both $f(\vw)$ and $h(\vv)$ are convex and smooth, and $\mA$ has full column rank. By assuming that we can access the extended first order oracle in \eqref{eq:extend-oracle-2}, we propose a bilinear accelerated extragradient (BAXG) method to solve this class of problems. We have shown that, to find an $\epsilon$-accurate solution that satisfies \eqref{eq:appro-weak}, the BAXG method needs at most $O(1/\sqrt{\epsilon}\log \frac{1}{\epsilon})$ number of accessing the extended first order oracle in \eqref{eq:extend-oracle-2}, which substantially improves the previous lower bound complexity results $O(1/\epsilon)$ of methods that only allow to access the first order oracle in \eqref{eq:oracle-2}.
Meanwhile, in the future it is also of interest to extend the BAXG method into the stochastic setting, the finite sum setting, or the nonconvex-nonconcave setting, which are more pertinent to problems that arise in contemporary machine/statistical learning.

\bibliography{ref3}
\bibliographystyle{alpha}

\clearpage
\appendix
\onecolumn

\section{Proof for Section \ref{sec:axg}}

\subsection{Proof of Lemma \ref{lem:dis-upper}}
\label{subsec:lem:dis-upper}
\begin{proof}
By the optimality condition of $\vz_k$, we have $\forall \vy \in\gX,$
\begin{eqnarray}
\Big\langle \sum_{i=1}^k a_i(\nabla g(\vx_i)+\mQ\hat{\vz}_i)+ 2(\vz_k - \vx_0), \vy - \vz_k\Big\rangle \ge 0. 
\end{eqnarray}
Then 
\begin{eqnarray}
\psi_k(\vz_k)&=&\sum_{i=1}^ka_i( g(\vx_i) + \langle \nabla g(\vx_i), \vz_k - \vx_i\rangle + \langle \mQ\hat{\vz}_i, \vz_k-\vy\rangle ) +  \|\vz_k-\vx_0\|^2\nonumber\\ 
&\le&\sum_{i=1}^ka_i(g(\vx_i) + \langle \nabla g(\vx_i), \vz_k - \vx_i\rangle) + \Big\langle \sum_{i=1}^k a_i \nabla g(\vx_i) +2(\vz_k - \vx_0), \vy - \vz_k\Big\rangle +   \|\vz_k-\vx_0\|^2\nonumber\\
&\le&\sum_{i=1}^ka_i ( g(\vx_i) +  \langle \nabla g(\vx_i), \vy - \vx_i\rangle )+\langle 2(\vz_k - \vx_0), \vy - \vz_k\rangle +   \|\vz_k-\vx_0\|^2\nonumber\\
&\le& \sum_{i=1}^ka_i g(\vy) +   \|\vy-\vx_0\|^2\nonumber\\
&=& A_k g(\vy)  +  \|\vy-\vx_0\|^2.\nonumber
\end{eqnarray}

Lemma \ref{lem:dis-upper} is proved.
\end{proof}

\subsection{Proof of Lemma \ref{lem:dis-lower}}
\label{subsec:lem:dis-lower}
\begin{proof}
First, in \eqref{eq:dis}, by $A_0 =0$ and $\vz_0=\vx_0,$ we have
\begin{equation}
A_0 f(\vx_0) - \psi_0(\vz_0)=0.\label{eq:A-0-x-0}
\end{equation}
By our assumption, $\hat{f}(\vx; \vx_i,\hat{\vz}_i,\vy)$ is a linear function of $\vx$ and $\|\vx - \vx_0\|^2$ is $2$-strongly convex. Therefore for all $\vx, \vy\in \gX$, it follows that 
\begin{align}
\psi_{i}(\vx)\ge \psi_{i}(\vy)+\langle \nabla \psi_{i}(\vy), \vx-\vy\rangle  +  \|\vx-\vy\|^2.
\end{align}
By the optimality condition $\vz_i$, we have $\forall \vx\in\gX, \langle \nabla \psi_{i}(\vz_i), \vx-\vz_i\rangle\ge 0.$ So it follows that 
\begin{eqnarray}
\psi_{i}(\vx)\ge \psi_i(\vz_i)  + \|\vx-\vz_i\|^2.\label{eq:axg-dis-tmp1}
\end{eqnarray}
Therefore applying \eqref{eq:axg-dis-tmp1} to $\psi_{i-1}(\vx)$, we have
\begin{eqnarray}
\psi_{i}(\vx) &=& \psi_{i-1}(\vx)+ a_{i}\hat{f}(\vx; \vx_{i},\hat{\vz}_i, \vy)\nonumber\\
&\ge& \psi_{i-1}(\vz_{i-1}) + \|\vx-\vz_{i-1}\|^2 + a_{i}\hat{f}(\vx;\vx_{i},\hat{\vz}_i, \vy)\nonumber\\
&=& \psi_{i-1}(\vz_{i-1}) + \|\vx-\vz_{i-1}\|^2 + 
a_i(g(\vx_i) + \langle \nabla g(\vx_i), \vx - \vx_i\rangle+ \langle \mQ\hat{\vz}_i, \vx-\vy\rangle).
\label{eq:axg-dis-tmp2}
\end{eqnarray}
Meanwhile, we can give a lower bound of the last term of RHS of \eqref{eq:axg-dis-tmp2} as follows
\begin{eqnarray*}
&& a_{i}(g(\vx_i) + \langle \nabla g(\vx_i), \vx-\vx_{i}\rangle  ) \\
&=&A_{i}\left(g(\vx_{i})  + \left\langle \nabla g(\vx_{i}), \frac{a_{i}}{A_{i}}\vx+\frac{A_{i-1}}{A_{i}}\vx_{i-1}-\vx_{i}\right\rangle\right) \\
&&\quad-A_{i-1}(g(\vx_{i})  + \langle \nabla g(\vx_{i}), \vx_{i-1}-\vx_{i}\rangle) \\
&{\ge}&A_{i}\left(g(\vx_{i})  + \left\langle \nabla g(\vx_{i}), \frac{a_{i}}{A_{i}}\vx+\frac{A_{i-1}}{A_{i}}\vx_{i-1}-\vx_{i}\right\rangle\right)-A_{i-1}g(\vx_{i-1}) \\
&=&A_{i}g(\vx_{i})-A_{i-1} g(\vx_{i-1}) + A_{i}  \left\langle \nabla g(\vx_{i}), \frac{a_{i}}{A_{i}}\vx+\frac{A_{i-1}}{A_{i}}\vx_{i-1}-\vx_{i}\right\rangle,
\end{eqnarray*}
and
\begin{eqnarray}
a_i \langle \mQ\hat{\vz}_i, \vx - \vy\rangle =  a_i \langle \mQ\hat{\vz}_i, \vx - \hat{\vz}_i\rangle +  a_i \langle \mQ\hat{\vz}_i, \hat{\vz}_i - \vy\rangle. 
\end{eqnarray}

Therefore, it follows that
\begin{eqnarray}
\psi_{i}(\vx) &\ge& \psi_{i-1}(\vz_{i-1}) + \|\vx-\vz_{i-1}\|^2 +  A_{i}g(\vx_{i}) -A_{i-1} g(\vx_{i-1}) \nonumber\\
&&+ A_{i}  \left\langle \nabla g(\vx_{i}), \frac{a_{i}}{A_{i}}\vx+\frac{A_{i-1}}{A_{i}}\vx_{i-1}-\vx_{i}\right\rangle+a_i \langle \mQ\hat{\vz}_i, \vx - \hat{\vz}_i\rangle +  a_i \langle \mQ\hat{\vz}_i, \hat{\vz}_i - \vy\rangle
. \label{eq:axg-es-2-1}
\end{eqnarray}

By setting $\vx\doteq\vz_{i}$ and a simple arrangement of \eqref{eq:axg-es-2-1}, we have
\begin{eqnarray}
&&(A_{i}g(\vx_{i})- \psi_{i}(\vz_{i}))-(A_{i-1} g(\vx_{i-1})- \psi_{i-1}(\vz_{i-1}))+ a_i \langle \mQ\hat{\vz}_i, \hat{\vz}_i -\vy\rangle \nonumber\\
&\le& A_{i}\left\langle \nabla g(\vx_{i}), \vx_{i}-\frac{a_{i}}{A_{i}}\vz_{i}-\frac{A_{i-1}}{A_{i}}\vx_{i-1}\right\rangle -    \|\vz_{i}-\vz_{i-1}\|^2 +a_i \langle \mQ\hat{\vz}_i,  \hat{\vz}_i-\vz_i\rangle\nonumber\\
&=& a_{i}\left\langle\nabla g(\vx_{i}),  \hat{\vz}_i-\vz_i\right\rangle -    \|\vz_{i}-\vz_{i-1}\|^2 +a_i \langle \mQ\hat{\vz}_i,  \hat{\vz}_i-\vz_i\rangle\nonumber\\
&=&a_{i}\left\langle\nabla g(\vx_{i}) + \mQ\hat{\vz}_i ,  \hat{\vz}_i-\vz_i\right\rangle -    \|\vz_{i}-\vz_{i-1}\|^2
.\label{eq:axg-es-2-2}
\end{eqnarray}

By the monotone property of $\mQ(\vx)$ and the definition of $a_i$ and $\hat{\vz}_i$, it follows that
\begin{eqnarray}
&&\sum_{i=1}^ka_i\langle \mQ\hat{\vz}_i, \hat{\vz}_i -\vy\rangle  \ge \sum_{i=1}^ka_i\langle \mQ(\vy), \hat{\vz}_i -\vy\rangle
= \Big\langle \mQ\vy, \sum_{i=1}^k a_i(\hat{\vz}_i -\vy)\Big\rangle=A_k \langle \mQ\vy, \vx_k -\vy\rangle. \label{eq:es-2-22}
\end{eqnarray}

Summing \eqref{eq:axg-es-2-2} from $i=1$ to $k$ and by \eqref{eq:es-2-22} and \eqref{eq:A-0-x-0}, it follows that
\begin{eqnarray}
&&A_{k}(g(\vx_k)+\langle \mQ\vy, \vx_k-\vy\rangle)- \psi_{k}(\vz_{k})\nonumber\\
&\le& A_{k}g(\vx_k)- \psi_{k}(\vz_{k}) + \sum_{i=1}^k a_i\langle \mQ\hat{\vz}_i, \hat{\vz}_i -\vy\rangle \nonumber\\ 
&\le& A_{0}f(\vx_{0})- \psi_{0}(\vz_{0}) 
+ \sum_{i=1}^{k}\left(a_{i}\left\langle\nabla g(\vx_{i}) + \mQ\hat{\vz}_i,  \hat{\vz}_i-\vz_i\right\rangle -    \|\vz_{i}-\vz_{i-1}\|^2
\right)\nonumber\\
&=&\sum_{i=1}^{k}\left(a_{i}\left\langle \nabla g(\vx_{i}) + \mQ\hat{\vz}_i ,  \hat{\vz}_i-\vz_i\right\rangle -    \|\vz_{i}-\vz_{i-1}\|^2\right).\nonumber
\end{eqnarray}

Hence, by the definition of $E_i$, Lemma \ref{lem:dis-lower} is proved.

\end{proof}

\subsection{Proof of Lemma \ref{lem:E-3}}\label{sec:lem:E-3}

\begin{proof}
According to the Step 7 and 8 of Algorithm \ref{alg:axg}, in the $i$-th iteration, we have
\begin{eqnarray}
\vx_i &=& \hat{\vx}_{i-1} + \frac{a_i}{A_i}(\hat{\vz}_i - \vz_{i-1}). 
\end{eqnarray}

By the definition of $E_i$ in Lemma \ref{lem:dis-lower}, one has: $\forall i\in[k],$
\begin{eqnarray}
E_{i}&\le&a_{i}\left\langle \nabla g(\vx_{i}) + \mQ\hat{\vz}_i ,  \hat{\vz}_i-\vz_i\right\rangle - \|\vz_{i}-\vz_{i-1}\|^2
\nonumber\\
&\le&  a_{i}\left\langle \nabla g(\vx_{i}) + \mQ\hat{\vz}_i + \frac{2}{a_i}(\hat{\vz}_{i} - \vz_{i-1}),\hat{\vz}_{i}-\vz_{i}\right\rangle   -   \left(  \|\hat{\vz}_{i} - \vz_{i-1}\|^2 +  \|\hat{\vz}_{i} - {\vz}_{i}\|^2\right) \nonumber\\
&\le&a_{i}\langle \nabla g(\vx_{i}) - \nabla g(  \hat{\vx}_{i-1}), \hat{\vz}_{i} - \vz_{i}\rangle
 + a_{i}\left\langle  \nabla {g}(  \hat{\vx}_{i-1})+ \mQ\hat{\vz}_i + \frac{2}{a_i}(\hat{\vz}_{i} - \vz_{i-1}),  \hat{\vz}_{i} - \vz_{i}\right\rangle \nonumber\\
 &&-    \left(  \|\hat{\vz}_{i} - \vz_{i-1}\|^2 +  \|\hat{\vz}_{i} - {\vz}_{i}\|^2\right) \nonumber\\
&\le& a_i\|\nabla g(\vx_{i}) - \nabla g(  \hat{\vx}_{i-1})\| \|\hat{\vz}_{i} - \vz_{i}\| + a_{i}\left\langle  \nabla {g}(  \hat{\vx}_{i-1})+ \mQ\hat{\vz}_i + \frac{2}{a_i}(\hat{\vz}_{i} - \vz_{i-1}),  \hat{\vz}_{i} - \vz_{i}\right\rangle \nonumber\\
&&- \left(  \|\hat{\vz}_{i} - \vz_{i-1}\|^2 +  \|\hat{\vz}_{i} - {\vz}_{i}\|^2\right) \nonumber\\
 &\le&a_iL\|\vx_i -  \hat{\vx}_{i-1}\|\|\hat{\vz}_{i} - {\vz}_{i}\|+ a_{i}\left\langle  \nabla {g}(  \hat{\vx}_{i-1})+ \mQ\hat{\vz}_i + \frac{2}{a_i}(\hat{\vz}_{i} - \vz_{i-1}),  \hat{\vz}_{i} - \vz_{i}\right\rangle \nonumber\\
&&- \left(  \|\hat{\vz}_{i} - \vz_{i-1}\|^2 +  \|\hat{\vz}_{i} - {\vz}_{i}\|^2\right) \nonumber\\
&=&\frac{a_i^2L}{A_i}\|\hat{\vz}_i - \vz_{i-1}\|\|\hat{\vz}_i  - \vz_i\| + a_i\Big\langle \nabla g(\hat{\vx}_{i-1}) + \mQ \vz_{i-1} + (\mQ + \frac{2}{a_i}\mI)(\hat{\vz}_i - \vz_{i-1}), \hat{\vz}_i  - \vz_i\Big\rangle  \nonumber\\ 
&&- \left(  \|\hat{\vz}_i - \vz_{i-1}\|^2 +  \|\hat{\vz}_i  - \vz_i\|^2\right) \nonumber\\
&\le& \frac{1}{2}\Big(\Big(\frac{a_i^2L}{A_i}\Big)^2 -1\Big)\|\hat{\vz}_i - \vz_{i-1}\|^2  + a_i\Big\langle \nabla g(\hat{\vx}_{i-1}) + \mQ \vz_{i-1} + (\mQ + \frac{2}{a_i}\mI)(\hat{\vz}_i - \vz_{i-1}), \hat{\vz}_i  - \vz_i\Big\rangle  \nonumber\\ 
&&- \frac{1}{2}\left(  \|\hat{\vz}_i - \vz_{i-1}\|^2 +  \|\hat{\vz}_i  - \vz_i\|^2\right). \label{eq:lem:E-311}
\end{eqnarray}

So by the Step 6 of Algorithm \ref{alg:axg}, in the $i$-th iteration, we have $\forall \vz \in \gX,$ 
\begin{eqnarray}
&&\!\!\!\!\!\!\!\!\!\!\!\! a_i\Big\langle \nabla g(\hat{\vx}_{i-1}) + \mQ \vz_{i-1} + (\mQ + \frac{2}{a_i}\mI)(\hat{\vz}_i - \vz_{i-1}), \hat{\vz}_i  - \vz\Big\rangle  
- \frac{1}{2}\left(  \|\hat{\vz}_i - \vz_{i-1}\|^2 +  \|\hat{\vz}_i  - \vz\|^2\right)\le 0. \label{eq:lem:E-312}
\end{eqnarray}
Therefore, by combining \eqref{eq:lem:E-311} and \eqref{eq:lem:E-312}, we have 
\begin{eqnarray}
E_i\le \frac{1}{2}\Big(\Big(\frac{a_i^2L}{A_i}\Big)^2 -1\Big)\|\hat{\vz}_i - \vz_{i-1}\|^2,
\end{eqnarray}
and Lemma \ref{lem:E-3} is proved.

\end{proof}

\section{Proof for Section \ref{sec:subprob}}

\subsection{Proof of Lemma \ref{lem:refor}}\label{sec:lem:refor}
\begin{proof}
In  the Step 6 of the $k$-th iteration of Algorithm \ref{alg:axg}, by the definitions of  $\{c, \vw^0,\vv^0,\vw^1,\vv^1, \tilde{\vw}, \tilde{\vv}\}$  in \eqref{eq:sim-notation2} and \eqref{eq:sim-notation}, we have 
\begin{align}
 \nabla g(\hat{\vx}_{k-1}) + \mQ \vz_{k-1} + (\mQ + \frac{2}{a_k}\mI)(\hat{\vz}_k - \vz_{k-1}) = 
 \left[\begin{matrix}
      \vw^0 + \mA(\tilde{\vv} - \vv^1) + \frac{2}{c}(\tilde{\vw}-\vw^1)\\
      \vv^0 -\mA^T(\tilde{\vw} - \vw^1) + \frac{2}{c}(\tilde{\vv}-\vv^1)
 \end{matrix}\right].
\end{align}
So the condition \eqref{eq:cond} is equivalent to 
$\forall \vw\in\bbR^d, \vv\in\gX,$ 
\begin{align}
&{c}\Big\langle \vw^0 + \mA(\tilde{\vv} - \vv^1) + \frac{2}{c}(\tilde{\vw} - \vw^1),   \tilde{\vw} -\vw\Big\rangle + {c}\Big\langle \vv^0 - \mA^T(\tilde{\vw} - \vw^1) + \frac{2}{c}(\tilde{\vv} - \vv^1),   \tilde{\vv} - \vv\Big\rangle + \nonumber\\
&\quad\quad-\frac{1}{2}\left(\|\vw - \vw^1\|^2 + \|\vv-\vv^1\|^2 + \| \tilde{\vw} - \vw\|^2 + \|\tilde{\vv}-\vv\|^2\right)\le 0. \label{eq:cond-refor}
\end{align}

By setting 
\begin{align}
\vw^0 + \mA (\tilde{\vv} - \vv^1) + \frac{2}{c} (\tilde{\vw} - \vw^1) = \vzero \quad   \Longleftrightarrow   \quad \tilde{\vw} = \vw^1 - \frac{c}{2} (\vw^0 + \mA(\tilde{\vv}-\vv^1)), \label{eq:cond-refor-2}
\end{align}
a sufficient condition for \eqref{eq:cond-refor} is 
\begin{eqnarray}
{c}\Big\langle \vv^0 - \mA^T(\tilde{\vw} - \vw^1) + \frac{2}{c}(\tilde{\vv} - \vv^1),   \tilde{\vv} - \vv\Big\rangle -\frac{1}{2}\left(\|\vv-\vv^1\|^2 + \|\tilde{\vv}-\vv\|^2\right)\le 0.\label{eq:cond-refor-3}
\end{eqnarray}

By setting $\forall \vv\in\gV,$
\begin{align}
l(\vv) \doteq& \frac{1}{2}(\vv-\vv^1)^T\big(\mA^T\mA + \frac{4}{c^2}\mI\big)(\vv-\vv^1) + \big( \frac{2}{c}\vv^0 + \mA^T\vw^0\big)^T(\vv-\vv^1), \label{eq:cond-refor-4}       
\end{align}
and by \eqref{eq:cond-refor-2}, \eqref{eq:cond-refor-3} is equivalent to 
\begin{align}
&\big\langle \nabla l(\tilde{\vv}), \tilde{\vv}-  {\vv}\big\rangle -\frac{1}{c^2}\left( \|\tilde{\vv} - \vv^1\|^2  + \| \tilde{\vv} - \vv\|^2\right)\le 0, \label{eq:cond-refor-33}   
\end{align}
To guarantee \eqref{eq:cond-refor-33}, by the Cauchy-Schwarz inequality $2ab\le a^2 +b^2$, a sufficient condition is $\forall \vv\in\gV,$
\begin{align}
\big\langle \nabla l(\tilde{\vv}), \tilde{\vv}-  {\vv}\big\rangle \le \frac{{2}}{c^2}\|\tilde{\vv} - \vv^1\| \| \tilde{\vv} - \vv\|.     
\end{align}

Lemma \ref{lem:refor} is proved.

\end{proof}

\subsection{Proof of Lemma \ref{lem:varphi-upper}}\label{sec:lem:varphi-upper}

\begin{proof}
It follows that $\forall \vu \in \gV,$ 
\begin{eqnarray}
\varphi_t(\vu)&=&\sum_{i=1}^t b_i( l(\hat{\vv}_{i-1}) + \langle \nabla l(\hat{\vv}_{i-1}), \vu - \hat{\vv}_{i-1}\rangle + \frac{\sigma_l}{2}\|\vu - \hat{\vv}_{i-1}\|^2) +  \frac{1}{2}\|\vu-\vv_0\|^2\nonumber\\ 
&\le& \sum_{i=1}^t b_i l(\vu) +  \frac{1}{2} \|\vu-\vv_0\|^2\nonumber\\
&=& B_t l(\vu)  +  \frac{1}{2}\|\vu-\vv_0\|^2.\nonumber
\end{eqnarray}

So we have 
\begin{eqnarray}
\varphi_t(\vu_t) =  \min_{\vu\in\gV}\varphi_t(\vu)\le \varphi_t(\vv^*)\le  B_t l(\vv^*)+ \frac{1}{2}\|\vv^*-\vv_0\|^2.
\end{eqnarray}

Lemma \ref{lem:varphi-upper} is proved.
\end{proof}

\subsection{Proof of Lemma \ref{lem:varphi-lower}}\label{sec:lem:varphi-lower}
\begin{proof}
By our setting that $B_0=0, \vu_0 = \vv_0$ in Algorithm \ref{alg:pagd}, we have
\begin{equation}
B_0 l(\vv_0) - \varphi_0(\vu_0)=0.\label{eq:var-A-0-x-0}
\end{equation}
Meanwhile for all $\vu, \vv\in \gV$ and $1\le i\le k$, it follows that 
\begin{align}
\varphi_{i}(\vu)\ge \varphi_{i}(\vv)+\langle \nabla \varphi_{i}(\vv), \vu-\vv\rangle  + \frac{1+\sigma_l B_{i-1}}{2}  \|\vu-\vv\|^2.
\end{align}
Then in the $i$-th iteration, by the optimality condition of $\vu_i$, we have $\forall \vu\in\gV, \langle \nabla \varphi_i(\vu_i), \vu-\vu_i\rangle \ge 0$. So it follows that 
\begin{eqnarray}
\varphi_{i}(\vu)\ge \varphi_i(\vu_i)  + \frac{1+\sigma_l B_{i}}{2}  \|\vu-\vu_i\|^2.\label{eq:dis-tmp1}
\end{eqnarray}
Therefore by using \eqref{eq:dis-tmp1} on $\varphi_{i-1}(\vu)$, we have
\begin{eqnarray}
\varphi_{i}(\vu) &=& \varphi_{i-1}(\vu)+ b_i(l(\hat{\vv}_{i-1}) + \langle \nabla l(\hat{\vv}_{i-1}), \vu - \hat{\vv}_{i-1}\rangle + \frac{\sigma_l}{2}\|\vu - \hat{\vv}_{i-1}\|^2)\nonumber\\
&\ge& \varphi_{i-1}(\vu_{i-1}) +  \frac{1+\sigma_l B_{i-1}}{2}\|\vu-\vu_{i-1}\|^2 + b_i(l(\hat{\vv}_{i-1}) + \langle \nabla l(\hat{\vv}_{i-1}), \vu - \hat{\vv}_{i-1}\rangle + \frac{\sigma_l}{2}\|\vu - \hat{\vv}_{i-1}\|^2)\nonumber\\
&\ge& \varphi_{i-1}(\vu_{i-1}) +  \frac{1+\sigma_l B_{i-1}}{2} \|\vu-\vu_{i-1}\|^2 + 
b_i(l(\hat{\vv}_{i-1}) + \langle \nabla l(\hat{\vv}_{i-1}), \vu - \vu_{i-1}\rangle).
\label{eq:dis-tmp2}
\end{eqnarray}

Let $\tilde{\vv}_i \doteq \frac{B_{i-1}}{B_i}\vv_{i-1} + \frac{b_i}{B_i}\vu_i.$ Then by the setting $\hat{\vv}_{i-1} = \frac{B_{i-1}}{B_i}\vv_{i-1} + \frac{b_i}{B_i}\vu_{i-1},$  we have $\tilde{\vv}_i - \hat{\vv}_{i-1} = \frac{b_i}{B_i}(\vu_i - \vu_{i-1}).$

By the definition of the sequences $\{b_i\}$ and $\{B_i\}$, we have
$\frac{(1+\sigma_l B_{i-1})B_i}{b_i^2}\ge 2L$. In \eqref{eq:dis-tmp2}, let $\vu \doteq  \vu_i$, then we have  
\begin{eqnarray}
\varphi_i(\vu_i) - \varphi_{i-1}(\vu_{i-1})&\ge&   \frac{1+\sigma_l B_{i-1}}{2} \|\vu_i-\vu_{i-1}\|^2 + 
b_i(l(\hat{\vv}_{i-1}) + \langle \nabla l(\hat{\vv}_{i-1}), \vu_i - \hat{\vv}_{i-1}\rangle)\nonumber\\
&=&  \frac{(1+\sigma_l B_{i-1})B_i^2}{2b_i^2} \|\tilde{\vv}_i - \hat{\vv}_{i-1}\|^2 + 
b_i(l(\hat{\vv}_{i-1}) + \langle \nabla l(\hat{\vv}_{i-1}), \vu_i - \hat{\vv}_{i-1}\rangle)\nonumber\\
&=& \frac{(1+\sigma_l B_{i-1})B_i^2}{2b_i^2} \|\tilde{\vv}_i - \hat{\vv}_{i-1}\|^2 + B_i(l(\hat{\vv}_{i-1}) + \langle \nabla l(\hat{\vv}_{i-1}), \tilde{\vv}_i - \hat{\vv}_{i-1}\rangle) \nonumber\\
&&- B_{i-1}(l(\hat{\vv}_{i-1}) + \langle \nabla l(\hat{\vv}_{i-1}),  \vv_{i-1} - \hat{\vv}_{i-1} \rangle) \nonumber\\ 
&=&B_i\Big(l(\hat{\vv}_{i-1}) + \langle \nabla l(\hat{\vv}_{i-1}), \tilde{\vv}_i - \hat{\vv}_{i-1}\rangle + \frac{(1+\sigma_l B_{i-1})B_i}{2b_i^2} \|\tilde{\vv}_i - \hat{\vv}_{i-1}\|^2 \Big) \nonumber\\
&&- B_{i-1}(l(\hat{\vv}_{i-1}) + \langle \nabla l(\hat{\vv}_{i-1}),  \vv_{i-1} - \hat{\vv}_{i-1} \rangle) \nonumber\\ 
&\ge& B_i\Big(l(\hat{\vv}_{i-1}) + \langle \nabla l(\hat{\vv}_{i-1}), \tilde{\vv}_i - \hat{\vv}_{i-1}\rangle + L \|\tilde{\vv}_i - \hat{\vv}_{i-1}\|^2 \Big) \nonumber\\
&&- B_{i-1}(l(\hat{\vv}_{i-1}) + \langle \nabla l(\hat{\vv}_{i-1}),  \vv_{i-1} - \hat{\vv}_{i-1} \rangle) \nonumber\\ 
&\ge& B_i\Big(l(\hat{\vv}_{i-1}) + \langle \nabla l(\hat{\vv}_{i-1}), {\vv}_i - \hat{\vv}_{i-1}\rangle + L \|{\vv}_i - \hat{\vv}_{i-1}\|^2 \Big) \nonumber\\
&&- B_{i-1}(l(\hat{\vv}_{i-1}) + \langle \nabla l(\hat{\vv}_{i-1}),  \vv_{i-1} - \hat{\vv}_{i-1} \rangle) \nonumber\\ 
&\ge& B_i\Big(l(\vv_i) + \frac{L}{2} \|{\vv}_i - \hat{\vv}_{i-1}\|^2 \Big) - B_{i-1}(l(\hat{\vv}_{i-1}) + \langle \nabla l(\hat{\vv}_{i-1}),  \vv_{i-1} - \hat{\vv}_{i-1} \rangle) \nonumber\\
&\ge&  B_i\Big(l(\vv_i) + \frac{L}{2} \|{\vv}_i - \hat{\vv}_{i-1}\|^2 \Big) -  B_{i-1} l(\vv_{i-1})\nonumber\\ 
&\ge&  B_i l(\vv_i) -  B_{i-1} l(\vv_{i-1})  + \frac{LB_i}{2} \|{\vv}_i - \hat{\vv}_{i-1}\|^2. \label{eq:dis-tmp222}
\end{eqnarray}
Telescoping \eqref{eq:dis-tmp222} from $i=1$ to $t$, we have
\begin{align}
\varphi_t(\vu_t) - \varphi_0(\vu_0) \ge    B_t l(\vv_t) -  B_{0} l(\vv_{0})  +\sum_{i=1}^t  \frac{LB_i}{2}\|{\vv}_i - \hat{\vv}_{i-1}\|^2. \label{eq:dis-tmp2221} 
\end{align}

Therefore, combining \eqref{eq:var-A-0-x-0} and \eqref{eq:dis-tmp2221}, we have  \begin{align}
 B_t l(\vv_t) + \sum_{i=1}^t \frac{LB_i}{2}\|{\vv}_i - \hat{\vv}_{i-1}\|^2 \le \varphi_t(\vu_t),   
\end{align}
and Lemma \ref{lem:varphi-lower} is proved.

\end{proof}

\subsection{Proof of Lemma \ref{lem:strong-dist}}\label{sec:lem:strong-dist}
\begin{proof}
In the $i$-th iteration of Algorithm \ref{alg:pagd}, by the optimality condition of $\vv_i$, we have: $\forall \vv\in\gV,$
\begin{align}
 \Big\langle \nabla l(\hat{\vv}_{i-1})+ 2L_l(\vv_i - \hat{\vv}_{i-1}), \vv_i - \vv\Big\rangle \le 0.
\end{align}

So it follows that $\forall \vv\in\gV,$
\begin{eqnarray}
\langle \nabla l(\vv_i), \vv_i - \vv\rangle &=& \Big\langle \nabla l(\vv_i) - (\nabla l(\hat{\vv}_{i-1})+ 2L_l(\vv_i - \hat{\vv}_{i-1})), \vv_i - \vv\Big\rangle \nonumber\\
&&+   \Big\langle \nabla l(\hat{\vv}_{i-1})+ 2L_l(\vv_i - \hat{\vv}_{i-1}), \vv_i - \vv\Big\rangle\nonumber\\
&\le& \Big\langle \nabla l(\vv_i) - (\nabla l(\hat{\vv}_{i-1})+ 2L_l(\vv_i - \hat{\vv}_{i-1})), \vv_i - \vv\Big\rangle \nonumber\\
&\le& \|\nabla l(\vv_i) - \nabla l(\hat{\vv}_{i-1})\|\|\vv_i - \vv\| + 2L_l\|\vv_i - \hat{\vv}_{i-1} \|\|\vv_i - \vv\| \nonumber\\
&=&L_l\|\vv_i - \hat{\vv}_{i-1}\|\|\vv_i - \vv\| + 2L_l\|\vv_i - \hat{\vv}_{i-1} \|\|\vv_i - \vv\|\nonumber\\
&=& 3L_l\|\vv_i - \hat{\vv}_{i-1}\|\|\vv_i - \vv\|.\label{eq:sd-1}
\end{eqnarray}

Let $\vv\doteq\vv^*$. Then by the optimality condition of $\vv^*$ such that $\langle\nabla l(\vv^*), \vv_i-\vv^*\rangle\ge 0$ and the strong convexity of $l(\vv)$, we have 
\begin{eqnarray}
\langle\nabla l(\vv_i), \vv_i-\vv^*\rangle \ge \langle\nabla l(\vv^*), \vv_i-\vv^*\rangle + {\sigma_l}\|\vv_i-\vv^*\|^2\ge {\sigma_l}\|\vv_i-\vv^*\|^2. \label{eq:sd-2}
\end{eqnarray}
So by \eqref{eq:sd-1} and \eqref{eq:sd-2}, we have
\begin{eqnarray}
\|\vv_i-\vv^*\| \le \frac{3L_l}{\sigma_l}\|\vv_i - \hat{\vv}_{i-1}\|. 
\end{eqnarray}

Lemma \ref{lem:strong-dist} is proved.
\end{proof}

\subsection{Proof of Theorem \ref{thm:subprob}}\label{sec:thm:subprob}
\begin{proof}
By combining Lemmas \ref{lem:varphi-lower}, \ref{lem:varphi-upper}, we have 
\begin{eqnarray}
B_t l(\vv_t) + \sum_{i=1}^t \frac{L B_i}{2}\|\vv_i - \hat{\vv}_{i-1}\|^2 \le \varphi_t(\vu_t)
\le B_t l(\vv^*)+ \frac{1}{2}\|\vv^*-\vv_0\|^2.
\end{eqnarray}

By the fact $l(\vv_t) \ge l(\vv^*)$, we have
\begin{align}
\sum_{i=1}^t \frac{L_lB_i}{2}\min_{i\in[t]}\|\vv_i - \hat{\vv}_{i-1}\|^2\le \sum_{i=1}^t \frac{L_lB_i}{2}\|\vv_i - \hat{\vv}_{i-1}\|^2 \le \frac{1}{2}\|\vv^*-\vv_0\|^2.  
\end{align}
So in the $t$ iterations, let $\tilde{i} \doteq \arg \min_{i\in[t]}\|\vv_i - \hat{\vv}_{i-1}\|.$   
\begin{eqnarray}
\|\vv_{\tilde{i}} - \hat{\vv}_{\tilde{i}-1}\|= \min_{i\in[t]}\|\vv_i - \hat{\vv}_{i-1}\|\le {\frac{\|\vv^*-\vv_0\|}{\sqrt{L_l\Big(\sum_{i=1}^t {B_i}\Big)}}}.\label{eq:thm-11}
\end{eqnarray}
Combining \eqref{eq:thm-11} and Lemma \ref{lem:strong-dist}, we have 
\begin{eqnarray}
\|\vv^* - \vv_0\| - \|\vv_{\tilde{i}}-\vv_0\|  \le \|\vv_{\tilde{i}}-\vv^*\| \le \frac{3L_l}{\sigma_l}\|\vv_{\tilde{i}} - \hat{\vv}_{{\tilde{i}}-1}\|\le \frac{3\sqrt{L_l}}{\sigma_l} 
\frac{ \|\vv^*-\vv_0\|}{\sqrt{\sum_{i=1}^t {B_i}}},
\end{eqnarray}
and this gives
\begin{eqnarray}
\|\vv^*-\vv_0\| \le \frac{\|\vv_{\tilde{i}}-\vv_0\|}{1- \frac{3\sqrt{L_l}}{\sigma_l\sqrt{\sum_{i=1}^t {B_i}}}}. \label{eq:thm-12}
\end{eqnarray}
Then combining \eqref{eq:thm-11} and \eqref{eq:thm-12} and with the setting $\vv_0 = \vv^1$ where $\vv^1$ is defined in \eqref{eq:sim-notation}, we have
\begin{eqnarray}
\langle \nabla l(\vv_{\tilde{i}}), \vv_{\tilde{i}} - \vv\rangle  &\le&3L_l\|\vv_{\tilde{i}} - \hat{\vv}_{{\tilde{i}}-1}\|\|\vv_{\tilde{i}} - \vv\| \nonumber\\
&\le& 3L_l {\frac{\|\vv^*-\vv_0\|}{\sqrt{L_l\Big(\sum_{i=1}^t {B_i}\Big)}}}  \|\vv_{\tilde{i}} - \vv\|\\
&\le& { \frac{3\sqrt{L_l}}{\sqrt{\sum_{i=1}^t {B_i}}}} 
\cdot \frac{\|\vv_{\tilde{i}}-\vv_0\|\|\vv_{\tilde{i}} - \vv\|}{1-\frac{3\sqrt{L_l}}{\sigma_l\sqrt{\sum_{i=1}^t {B_i}}}}\nonumber\\
&=& { \frac{3\sqrt{L_l}}{\sqrt{\sum_{i=1}^t {B_i}}}} 
\cdot \frac{\|\vv_{\tilde{i}}-\vv^0\|\|\vv_{\tilde{i}} - \vv\|}{1-\frac{3\sqrt{L_l}}{\sigma_l\sqrt{\sum_{i=1}^t {B_i}}}}
.\label{eq:thm-13}
\end{eqnarray}
By \eqref{eq:thm-13}, to have $\langle \nabla l(\vv_i), \vv_i - \vv\rangle \le \frac{1}{c^2}\|\vv_i-\vv^0\|\|\vv_i - \vv\|$, it is easy to check that a sufficient condition is 
\begin{eqnarray}
\sum_{i=1}^t B_i \ge  \Big( \frac{3\sqrt{L_l}(\sigma_l c^2 + 1)}{\sigma_l}\Big)^2.\label{eq:thm-14}
\end{eqnarray}

Meanwhile, with our setting of $\{b_i\}, \{B_i\}$ in Algorithm \ref{alg:pagd}, we have
\begin{eqnarray}
\sum_{i=1}^t B_i &=&\frac{1}{\sqrt{2L_l\sigma_l}}\left(\left(1+\sqrt{\frac{\sigma_l}{2L_l}}\right)^t-1\right). \label{eq:thm-15}
\end{eqnarray}
So for the sufficient condition \eqref{eq:thm-15} to be true, by \eqref{eq:thm-15} and the fact $\sigma_l =\sigma_{\min}^2 + \frac{1}{c^2} , L_l=\sigma_{\max}^2 + \frac{1}{c^2}$ with $c = a_k = \frac{2k-1}{L}$, we only need 
\begin{eqnarray}
t\ge \frac{\log\Big(1+9\sqrt{2}L_l^{\frac{3}{2}}\sigma_l^{-\frac{3}{2}}(\sigma_l c^2 + 1)^2  \Big)}{\log\Big(1+\sqrt{\frac{\sigma_l}{{2}L_l}}\Big)} 
= \tilde{O}\Big(\frac{\sigma_{\max}}{\sigma_{\min}}\log k \Big),
\end{eqnarray}
iterations, where $\tilde{O}$ hides the logarithmic factors about $\sigma_{\min},\sigma_{\max}$ and $L.$ 

Theorem \ref{thm:subprob} is proved.
\end{proof}

\end{document}